\documentclass[reqno,a4paper,11pt]{amsart}
\usepackage{amssymb,amsmath,amscd,amstext,amsthm,amsfonts}
\usepackage[]{graphicx}
\usepackage{epstopdf}
\usepackage{setspace} 
\usepackage{subfig}
\usepackage{color}
\usepackage[mathscr]{eucal}
\usepackage[draft]{changes}
\usepackage[ansinew]{inputenc} 
\usepackage{hyperref}
\usepackage{color}

\numberwithin{equation}{section}

\newtheorem{theorem}{Theorem}[section]
\newtheorem{proposition}[theorem]{Proposition}

\newtheorem{corollary}[theorem]{Corollary}
\newtheorem{lemma}[theorem]{Lemma}

{\theoremstyle{definition}
{
\newtheorem{remark}[theorem]{Remark}

\newtheorem{defn}[theorem]{Definition}
}}

\newcommand{\cal}{\mathcal}

\newcommand{\BB}{{\cal B}}

\newcommand{\EE}{{\cal E}}
\newcommand{\FF}{{\cal F}}
\newcommand{\GG}{{\cal G}}
\newcommand{\HH}{{\cal H}}

\newcommand{\JJ}{{\cal J}}

\newcommand{\MM}{{\cal M}}

\newcommand{\OO}{{\cal O}}

\newcommand{\RR}{{\cal R}}
\newcommand{\SSS}{{\cal S}}

\newcommand{\XX}{{\cal X}}
\newcommand{\YY}{{\cal Y}}

\newcommand{\Aa}{{\mathbb{A}}}

\newcommand{\Cc}{{\mathbb{C}}}

\newcommand{\Ii}{{\mathbb{I}}}

\newcommand{\Nn}{{\mathbb{N}}}

\newcommand{\Rr}{{\mathbb{R}}}

\newcommand{\Tt}{{\mathbb{T}}}

\newcommand{\Vv}{{\mathbb{V}}}

\newcommand{\Zz}{{\mathbb{Z}}}

\def\dist{\operatorname{dist}}

\def\Mat{\operatorname{Mat}}
\def\GL{\operatorname{GL}}

\def\S{\operatorname{S{}}}
\def\Sp{\operatorname{Sp}}

\def\supp{\operatorname{supp}}

\def\Per{\operatorname{Per}}

\newcommand{\id}  {\operatorname{id}}

\newcommand{\Int} {\operatorname{int}}

\newcommand{\Sym} {\operatorname{Sym}}

\newcommand{\te}[1]{\quad\text{#1}\quad}

\newcommand{\emb}[1]{\BB^{#1}}

\newcommand{\ff}{\mathbb{II}}

\newcommand{\bepsilon}{\boldsymbol{\varepsilon}}

\begin{document}
\title[Generic billiards on bodies]{Billiards in generic convex bodies have positive topological entropy}

\date{\today}

\author[Bessa]{M\'{a}rio Bessa}
\address{Universidade Aberta, Departamento de Ci\^encias e Tecnologia
\\
 Rua do Amial 752, 
 4200-055 Porto, Portugal}
\email{mario.costa@uab.pt}

\author[Del Magno]{Gianluigi Del Magno}
\address{Dipartimento di Matematica, Universit\`a di Pisa \\
Largo Bruno Pontecorvo 5, 56127 Pisa, Italy}
\email{gianluigi.delmagno@unipi.it}

\author[Lopes Dias]{Jo\~{a}o Lopes Dias}
\address{Departamento de Matem\'atica, CEMAPRE and REM, ISEG, 
Universidade de Lisboa,
Rua do Quelhas 6, 
1200-781 Lisboa, Portugal}
\email{jldias@iseg.ulisboa.pt}

\author[Gaiv\~ao]{Jos\'e Pedro Gaiv\~ao}
\address{Departamento de Matem\'atica, CEMAPRE and REM, ISEG\\
Universidade de Lisboa\\
Rua do Quelhas 6, 1200-781 Lisboa, Portugal}
\email{jpgaivao@iseg.ulisboa.pt}

\author[Torres]{Maria Joana Torres}
\address{CMAT and Departamento de Matem\'atica, 
Universidade do Minho, 
Campus de Gualtar, 
4700-057 Braga, Portugal}
\email{jtorres@math.uminho.pt}

\begin{abstract}
We show that there exists a $C^2$ open dense set of convex bodies with smooth boundary whose billiard map exhibits a non-trivial hyperbolic basic set. As a consequence billiards in generic convex bodies have positive topological entropy and exponential growth of the number of periodic orbits.
\end{abstract}

\maketitle

\section{Introduction}\label{sec:introduction}

We are interested in the billiard problem inside convex bodies in $\Rr^{d+1}$, $d\geq1$, with a smooth boundary.
The unit speed flow dynamics of the pointsized billiard ball is reduced to the billiard map. This map relates consecutive elastic reflections at the boundary of the body.
In this work, we present a way to perturb the body so that the billiard dynamics becomes chaotic on a proper subset of the phase space.

We naturally associate a convex body to its boundary, which is the embedding of a $d$-sphere $S$. Thus we talk about the topology on bodies as the one on embeddings $S\hookrightarrow\Rr^{d+1}$.
Moreover, we say that a convex body with a $C^\infty$ boundary is a smooth convex body.

Given a body whose billiard map is a diffeomorphism $f$,  an $f$-invariant and compact subset $\Lambda$ of the phase space is
\emph{hyperbolic} if there is a  $Df$-invariant continuous splitting of the tangent bundle restricted to $\Lambda$, $T_\Lambda M=E_\Lambda\oplus F_\Lambda$,
and $m\in{\mathbb{N}}$ such that for all  $x\in{\Lambda}$ the
following inequalities hold:
\begin{equation}\label{uh}
\|Df^{m}(x)|E_{x}\|\leq{\frac{1}{2}}
\te{and}
\|Df^{-m}(x)|F_{x}\|\leq{\frac{1}{2}}.
\end{equation}

A nontrivial hyperbolic basic set $\Lambda$ is a hyperbolic, infinite, transitive (contains a dense orbit) and locally maximal set, i.e. there is an open neighborhood $V$ of $\Lambda$ such that 
$$
\Lambda=\bigcap_{n\in\Zz}f^n(\bar V)
$$
(cf.~\cite{Katok}).
Such a set contains a transverse homoclinic point.

Our main result, proved in section~\ref{sec: proof}, is the following.

\begin{theorem}\label{main thm}
There is a $C^2$-open and dense set of smooth convex bodies 
whose billiard maps have a nontrivial hyperbolic basic set.
\end{theorem}

The complexity of the billiard dynamics can be
measured by the topological entropy, a numerical invariant which we now define. 
Let $d_n(x,y)=\max \{\dist(f^i(x),f^i(y))\colon 0\leq i <n-1\}$ where $f$ is the billiard map and $\dist$ is the distance induced by the embedding corresponding to the smooth convex body. 
A subset $F$ of the phase space is said to be $(n,\epsilon)$-spanning if the whole phase space is covered by the union of the dynamical balls $\{y\colon d_n(x,y)<\epsilon\}$ centered at the points $x\in F$. 
Denote by $N(n,\epsilon)$ the minimal cardinality of a $(n,\epsilon)$-spanning set.
Roughly, this gives the number of orbit segments that one can distinguish up to some precision. 
The topological entropy is then the exponential growth rate of this number as the precision increases,
$$
h_{top}(f)=\underset{\epsilon\rightarrow 0}{\lim}\left(\underset{n\rightarrow \infty}{\limsup}\frac{1}{n}\log N(n,\epsilon)\right).
$$

Since a nontrivial hyperbolic basic set $\Lambda$ contains a transverse homoclinic point, the topological entropy is positive.
Recall also that the topological entropy gives the exponential growth of the number of periodic orbits in $\Lambda$ (cf. ~\cite[Theorem 18.5.1]{Katok}).
Therefore, we have the following consequence of Theorem~\ref{main thm}.

\begin{corollary}
There is a $C^2$-open and dense set of smooth convex bodies whose billiard maps satisfy
$$
\limsup_{n\to\infty}\frac1n\log P_n=h_{top}(f)>0,
$$
where $P_n$ is the number of periodic orbits with period $n$.
\end{corollary}

Birkhoff in ~\cite{Birkhoff1927} was already interested in the problem of estimating from below the number of periodic orbits.
The exponential growth that we have obtained improves the best presently known estimate~\cite{karasev2009} (see also~\cite{BPHT2020}), when restricting to generic smooth convex bodies.

The proof of Theorem~\ref{main thm} is split mainly in the next two theorems, which are of independent interest. 
Define the integer
$$
m_d:=4\left(\begin{smallmatrix}2d+3\\4\end{smallmatrix}\right).
$$

\begin{theorem}\label{thm: 1}
There is a $C^\infty$-residual set $\RR$ of smooth convex bodies such that the following holds. If a billiard inside a body in $\RR$ has an elliptic periodic point with period $\geq m_d$, then it has a horseshoe.
\end{theorem}

\begin{theorem}\label{thm: 2}
There is a $C^\infty$-residual set $\RR$ of smooth convex bodies such that the following holds. If a billiard inside a body in $\RR$ has all periodic points with period $\geq m_d$ hyperbolic and this property is $C^2$-stable, then the closure of the set of those hyperbolic points contains a nontrivial hyperbolic basic set.
\end{theorem}

The above theorems are restated in Theorem~\ref{thm: ell} and Theorem~\ref{thm: hyp}, where all the relevant definitions and mathematical objects are thoroughly explained.

The openness property in Theorem~\ref{main thm} follows immediately from the structural stability of hyperbolic basic sets. 
To show denseness, which is the non-trivial part of the proof, we first observe that, among periodic points of sufficiently large period, the billiard inside a generic body either has an elliptic periodic point (A) or all periodic points are hyperbolic (B). 
In case (A) we are in the condition of Theorem~\ref{thm: 1} and there is a horseshoe.
On the other hand, in case (B),  we split again the proof in two alternatives (B1) and (B2) according to whether or not the hyperbolicity of all periodic points is preserved under $C^2$-perturbations.
For (B1) we use Theorem~\ref{thm: 2}, whereas in the  case (B2)
we  $C^2$-approximate the body by another body verifying (A).

It is worth pointing out that  it is an open problem whether or not there exist bodies as in Theorem~\ref{thm: 2} for which all periodic points of period $\ge m_d$ are hyperbolic.
The reason for this theorem is to deal with the possibility of having such billiards in the proof of Theorem~\ref{main thm}.

The proofs of Theorems~\ref{thm: 1} and ~\ref{thm: 2} require several perturbation lemmas in the context of billiards in bodies.
The most important ones are the versions of the Klingenberg-Takens theorem (Theorem~\ref{thm: K-T}) and the Franks' lemma (Theorem~\ref{Franks}).
These are new results that can be applied in further problems.

It is well-known that the dynamics around elliptic periodic orbits depends on high order derivatives of the map. 
In particular, a higher dimensional generalization of the twist property for area-preserving maps, the weakly monotonous property, requires up to the third derivative of the map.
The Klingenberg-Takens theorem for geodesic flows~\cite{KT-1972} gives a way to perturb metrics so that the jets of the Poincar\'e maps of closed orbits are inside a given invariant open dense set. We prove here the billiards version of this result (see ~\cite{carballo-miranda-2013} for the Tonelli Hamiltonians case). 
The local behaviour of weakly monotonous elliptic points plays a major role in the proof of Theorem~\ref{thm: 1}.

The Franks' lemma first appeared in~\cite[Lemma 1.1]{Franks1971} stating that perturbations of the derivative of diffeomorphisms at a finite set are indeed also derivatives of a $C^1$-close diffeomorphism (note that the result is no longer true for the $C^2$-topology~\cite{PS2009}).
This lemma is an essential tool to prove a variety of important and fundamental results on the stability and generic theories of dynamical systems displaying properties such as shadowing, structural stability, topological stability and expansiveness. 
Versions of the Franks lemma for more restricted classes of dynamical systems are available for flows~\cite{MPP2004,BGV}, volume-preserving diffeomorphisms~\cite{BDP2003}, divergence-free flows~\cite{BR}, symplectomorphisms and Hamiltonian flows~\cite{A-LD2014}.
The version for geodesic flows due to Contreras~\cite{contreras-am-2010} is more difficult since the $C^2$-perturbations are performed on the metric, so not local in the phase space.
Further extensions are available in~\cite{LRR2016}.
Here we present the billiards in bodies case generalizing the version for planar billiards~\cite{visscher-2015}.
Notice that the $C^2$-perturbations of the bodies are also not local in the phase space.
The ability to realize the perturbation of the tangent map is a crucial part of the proof of Theorem~\ref{thm: 2}.

For the particular case of planar billiards ($d=1$), it is known for  $r\geq3$ that a $C^r$-generic  convex domain  has a horseshoe~\cite{Cheng2004,CKP2007}. The proof is based on variational methods for two-dimensional twist maps which do not extend directly to higher dimensional billiards (see~\cite{Bialy2009} for an application of some variational properties in multidimensional billiards).  
For $d\geq1$, several properties related to periodic orbits are known. In particular, the existence of infinitely many periodic orbits ~\cite{farber-2002,farber-tabachnikov-2002,karasev2009}, and some generic properties due to Petkov and Stojanov ~\cite{stojanov-ETDS-1987,petkov-stojanov-mathZ-1987,petkov-stojanov-AJM-1987,petkov-ETDS-1988} (see section~\ref{sec: generic properties}).

In section~\ref{sec:coordinates} we introduce the basic setup concerning billiards in bodies, including the computation of the billiard map and its derivative. 
Theorem~\ref{main thm} is proved in section~\ref{sec: proof}.
The proof follows from Theorems~\ref{thm: ell} and ~\ref{thm: hyp}, corresponding to Theorems~\ref{thm: 1} and ~\ref{thm: 2}, respectively.
The former is proved in section ~\ref{sec: ellipt}, being crucial the version of the Klingenberg-Takens theorem given in section~\ref{sec: KT} and also a multidimensional version of a perturbation by Donnay~\cite{Donnay-2005} (Theorem~\ref{thm: Donnay}) that creates a transversal heteroclinic intersection.
The latter is proved in section~\ref{sec: hyp}, where it is required the use of the Franks' lemma version for billiards on bodies included in section~\ref{sec: Franks}.

\section{Billiard map}\label{sec:coordinates}

\subsection{Smooth convex bodies}

Let $r\in\{2,3,\ldots,\infty\}$, the $d$-sphere $S$, $d\geq1$, and the set of $C^r$ maps $S\to \Rr^{d+1}$ denoted by 
$
C^r(S,\Rr^{d+1}).
$
The subset of maps that are embeddings (i.e. $C^r$-diffeomorphisms onto its image) is written as
$$
C_{emb}^{r}(S,\Rr^{d+1}).
$$

The image of an embedding of the sphere is a $d$-$C^r$-submanifold of $\Rr^{d+1}$.
Given $\phi\in C_{emb}^{r}(S,\Rr^{d+1})$, we denote by $Q_\phi$ the bounded set whose boundary is $\Gamma_\phi := \phi(S)$.

Throughout, we fix a finite atlas $\{(\varphi_i,U_i)\}_{i\in I}$ of $S$. Let $p\in \Gamma_\phi$ and denote by $(\varphi,U)$ a chart of $S$ for which $p\in \phi(U)$. 
The tangent vectors 
\begin{equation}\label{eq: tg vectors}
t_j=t_j(p) = D(\phi\circ \varphi^{-1})(\varphi\circ \phi^{-1}(p) )\, e_j,
\quad
j=1,\dots,d,
\end{equation} 
define a basis for the tangent space $T_p\Gamma_\phi$, where $\{e_1,\dots,e_d\}$ denotes the canonical basis of $\Rr^d$. 
In addition,
\begin{eqnarray*}
Dt_i(p)\,t_j(p)
&=&
\left.\frac{d}{dr}\right|_{r=0}\left[t_j\circ \varphi\circ \phi^{-1}(re_i+\varphi\circ \phi^{-1}(p)\right]
\\
&=&
(\partial_i\partial_j(\phi_k\circ \varphi^{-1})(\varphi\circ \phi^{-1}(p) ))_{k=1}^d.
\end{eqnarray*}
So, we are able to deduce the following relation:
\begin{equation}\label{eq: Dt_j}
Dt_i(p)\,t_j(p)=Dt_j(p)\,t_i(p).
\end{equation}

By using the Euclidean inner product $\langle\cdot,\cdot\rangle$ and its corresponding norm $\|\cdot\|$, we denote by $N_\phi(p)\in\Rr^{d+1}$ the unit normal vector of $\Gamma_\phi$ at $p$ which is inward-pointing in $Q_\phi$. 
Since $N_\phi(p)$ has unit length, the derivative  $DN_\phi(p)$ maps the tangent space $T_p\Gamma_\phi$ into itself. 
This follows by differentiating $\langle N_\phi,N_\phi\rangle=1$ along $u\in T_p\Gamma_\phi$ at $p$, so that $\langle DN_\phi(p)\,u,N_\phi(p)\rangle=0$.

The shape operator of $\Gamma_\phi$ at $p$ is the linear map $L_\phi(p)\colon T_p\Gamma_\phi\to T_p\Gamma_\phi$ defined by 
$$
L_\phi(p)u=-DN_\phi(p)u.
$$ 
The first fundamental form of $\Gamma_\phi$ (or of $\phi$) at $p$ is the inner product restricted to the tangent space, i.e.
$$
\Ii_\phi(p)(u,v)=\langle u,v\rangle,
\quad
u,v\in T_p\Gamma_\phi
$$
and the second fundamental form of $\Gamma_\phi$ at $p$ is the bilinear map
$$
\ff_\phi (p)(u,v)=\langle L_\phi(p)u,v\rangle,
\quad
u,v\in T_p\Gamma_\phi.
$$
By differentiating $\langle N_\phi,t_j\rangle=0$ along $t_i$ at $p$ one gets the matrix representation of the second fundamental form in the basis $\{t_j\}_{j=1}^d$:
\begin{equation}\label{eq: 2nd fund form coeffs}
\ff_\phi(p)(t_i,t_j)=\langle Dt_i(p)\,t_j(p),N_\phi(p)\rangle.
\end{equation}
From~\eqref{eq: Dt_j} the symmetry of $\ff_\phi(p)$ follows, i.e. $\ff_\phi(p)(u,v)=\ff_\phi(p)(v,u)$, and the shape operator $L_\phi(p)$ is therefore self-adjoint.

Recall that a convex body is a convex, compact with non-empty interior subset of $\Rr^{d+1}$. We call it $C^r$-smooth if its boundary is the image of an embedding in $C_{emb}^r(S,\Rr^{d+1})$.

The set $Q_\phi$  is a smooth convex body iff $\ff_\phi(p)(u,u)\geq0$ for every $p\in \Gamma_\phi$ and $u\in T_p \Gamma_\phi$. 
The corresponding class of convex embeddings is denoted by 
$$
\XX^r \subset C_{emb}^{r}(S,\Rr^{d+1}).
$$

We denote the subset of the embeddings of the sphere corresponding to boundaries of bodies satisfying $\ff_\phi(p)(u,u)>0$ for every $p\in \Gamma_\phi$ and $u\in T_p \Gamma_\phi$, by
$$
\emb r \subset \XX^r.
$$
Notice that these are strictly convex bodies.

These spaces of embeddings are identified with the corresponding spaces of bodies in $\Rr^{d+1}$.

We will often drop the subscript $\phi$ to simplify notations.

\subsection{$C^r$-topology}

Recall the Whitney $C^r$ topology for $r\in\Nn$ given by the norm
$$
\|\phi\|_{C^r}= \max_{0\leq j\leq r} \max_{i\in I} \max_{y\in \varphi_i(U_i)} \|D^j (\phi\circ\varphi_i^{-1})(y)\|
$$
for any $\phi\in C^r(S,\Rr^{d+1})$.
This makes $C^r(S,\Rr^{d+1})$ a Baire space.
The union of the $C^r$-open sets of $C^\infty(S,\Rr^{d+1})$ for $r\in\Nn$ form a basis for the Whitney $C^\infty$-topology, making this also a Baire space.

Furthermore, $\emb r$ is $C^r$-open in $C^r(S,\Rr^{d+1})$ and also a Baire space for any $r\in\Nn\cup\{\infty\}$.

Clearly, $\emb r$ is $C^r$ open and dense in $\XX^r$.

\subsection{Perturbing the body} 

We consider here perturbations of  convex bodies along the normal at a boundary point. We also relate the respective second fundamental forms.

\begin{lemma}\label{lem:perturbing body}
Let $\phi\in \emb{r}$,  $r\in\{2,3,\ldots,\infty\}$, 
$p_0\in \Gamma:=\phi(S)$,
$s_0=\phi^{-1}(p_0)$
 and $\psi\colon S\to\Rr$ be $C^r$ such that $\psi(s_0)=0$ and 
 $D\psi(s_0)=0$. 
 If $\|\psi\|_{C^r}$ is sufficiently small, then 
\begin{enumerate}
\item $\tilde{\phi}:=\phi + \psi \cdot N(p_0)\in \emb r$,
\item $\|\tilde{\phi}-\phi\|_{C^\ell}\leq \|\psi\|_{C^\ell}$  for every $2\leq \ell\leq r$,
\item $\ff_{\tilde\phi}(p_0)=\ff_\phi(p_0)+D^2(\psi\circ\phi^{-1})(p_0)$.
\end{enumerate}
\end{lemma}

\begin{proof}
Denote by $\tilde{N}$ the unit normal vector field of $\tilde{\Gamma}$ where $\tilde{\Gamma}:=\tilde\phi(S)$, and $\{\tilde{t_j}(p)\}_{j=1}^d$ the basis of $T_{p}\tilde\Gamma$, $p\in\tilde \Gamma$, as in~\eqref{eq: tg vectors} now for $\tilde{\phi}$.
Clearly, $p_0=\tilde{\phi}\circ \phi^{-1}(p_0) \in\tilde{\Gamma}$, $\tilde{t_j}(p_0)=t_j(p_0)$ and $\tilde{N}(p_0)=N(p_0)$.
Hence,
\begin{eqnarray*}
D(\tilde t_j-t_j)(p_0)\,t_i(p_0)
&=&
D^2(\psi \circ \varphi^{-1})(\varphi\circ\phi^{-1}(p_0))\,(e_i,e_j)\,N(p_0)\\
&=&
D^2 (\psi\circ\phi^{-1}) (p_0) \,(t_i,t_j)\,N(p_0).
\end{eqnarray*}
The last equality comes from the computation of $D^2 (\psi\circ\phi^{-1}\circ\phi\circ\varphi^{-1}) (\varphi\circ\phi^{-1}(p_0))$ taking into account that $D\psi(s_0)=0$.

Finally, by~\eqref{eq: 2nd fund form coeffs} we get
\begin{eqnarray*}
\ff_{\tilde \phi}(p_0)(t_i,t_j)-\ff_\phi(p_0)(t_i,t_j) 
&=&
\langle D(\tilde t_i-t_i)(p_0)\,t_j(p_0),N(p_0)\rangle \\
&=&
D^2(\psi\circ\phi^{-1})(p_0) \, (t_i,t_j).
\end{eqnarray*}
For a small $\|\psi\|_{C^r}$ this implies that $\tilde\phi\in\emb r$.
\end{proof}

\subsection{The billiard map}

Let $\phi \in \emb{r}$. A convex body $Q_\phi$ whose boundary $\Gamma_\phi$ is the image of $\phi\in \emb r$ is called a billiard domain. 
The corresponding billiard is the flow on the unit tangent bundle of $Q_\phi$ generated
by the motion of a free point-particle inside the body with specular reflection
at $\Gamma_\phi$, i.e. the angle of reflection equals the angle of incidence. The billiard map $f_\phi$ is the first return map on $M_\phi$, the set of
unit vectors attached to $\Gamma_\phi$ and pointing inside $Q_\phi$.
More precisely, $f_\phi$ is a $C^{r-1}$-diffeomorphism (it has no singularities because the second fundamental form of $\partial Q_\phi$ is positive definite, see \cite{Kov88}) on the $2d$-dimensional manifold
$$
M_\phi=\left\{(p,v)\in \Gamma_\phi\times \Rr^{d+1}\colon \|v\|=1,\, \langle v,N(p)\rangle \geq 0\right\}.
$$

We will frequently refer to the dynamics of the billiard map $f_\phi$ on $M_\phi$ (or simply on $\phi$) as the dynamics of $\phi$.
In the following we omit the dependency of $\Gamma_\phi$, $M_\phi$ and $f_\phi$ on $\phi$.  We also define the projection $\pi_1\colon M\to \Gamma$ by $(p,v)\mapsto p$. The orbit of a point $x\in M$ under $f$ will be denoted by $O(x)$.

The free flight time $\tau\colon M \to \Rr$ between consecutive collisions on $\Gamma$ is defined as
$$
\tau(p,v)=
\begin{cases}
0,& v\perp N(p)\\
\min\{t>0\colon p+tv \in \Gamma\}, & \text{otherwise.}
\end{cases}
$$

Let $(p,v)\in M$ and $(\bar p,\bar v)=f(p,v)$. Then
$$
\begin{cases}
\bar{p}=p+\tau(p,v)v,\\
\bar{v}=R_{\bar{p}}v,
\end{cases}
$$
where $R_{\bar{p}}$ is the reflection in $T_{\bar{p}}\Gamma$, i.e.
$$
R_{\bar{p}}v=v-2\langle v,N(\bar{p})\rangle N(\bar{p}).
$$
The reflection yields that the vector $\bar v+v$ is tangent to $\Gamma$ at $\bar p$, as given by
\begin{equation}\label{eq: bar v-v}
\langle \bar v+v,N(\bar p)\rangle=0.
\end{equation}
Notice that $(p,v)$ is a fixed point of $f$ whenever $v\perp N(p)$.

\subsection{The derivative of the billiard map}

We introduce a new set of coordinates on $TM$ for which the derivative $Df$ has a convenient form. These coordinates are called Jacobi coordinates, induced by the so-called transversal Jacobi fields~\cite[Appendix B]{Wojtkowski-1988}.
To make the exposition self-contained we present here all the details.  In the following, we write $\OO(\varepsilon)$ for the usual big-O notation, i.e. a quantity that is uniformly bounded in norm by $const\cdot \varepsilon$  as $\varepsilon\to 0$.

Define the set of billiard directions at a point $p\in\Gamma$ by
$$
V_p=\{u\in\Rr^{d+1}\colon \|u\|=1, \langle u,N(p)\rangle >0 \}.
$$

Fix now $(p,v) \in M$ such that $v\in V_p$.
Since the billiard domain is strictly convex, we have $(\bar{p},\bar{v})=f(p,v)$ satisfies $\bar v\in V_{\bar p}$.

Notice that
$$
T_{(p,v)} M = N(p)^\perp\times v^\perp,
$$ 
where $N(p)^\perp$ and $v^\perp$ denote the hyperplanes in $\Rr^{d+1}$ that are orthogonal to $N(p)$ and $v$, respectively. 

\subsubsection{Projections}

Let $P_{v}$ be the orthogonal projection onto $v^\perp$, i.e.
$$
P_{v }\xi=\xi-\langle \xi,v\rangle v,\quad \xi\in\Rr^{d+1}. 
$$
For a unit vector $\eta\in\Rr^{d+1}$ which is not orthogonal to $v$, define also the projection along the direction $v$ onto the hyperplane $\eta^\perp$ by
$$
P_{\eta}^vu=u-\frac{\langle u,\eta\rangle}{\langle v,\eta\rangle}v,\quad u\in \Rr^{d+1}.
$$
It is simple to check that the adjoint of $P_\eta^v$ is $P_v^\eta$, i.e. $(P_{\eta}^v)^* = P_{v}^\eta$, and that
\begin{equation}\label{eq: proj eta' proj eta}
P_{\eta'}^v\circ P_\eta^v=P_{\eta'}^v
\end{equation}
where $\eta,\eta'$ are both unit vectors non-orthogonal to $v$.
Notice also that
\begin{equation}\label{eq: proj eta v proj v}
P_\eta^v\circ P_v|_{\eta^\perp}=I.
\end{equation}

Recall the definition of the billiard map.
The reflection $R_{\bar{p}}$ identifies isometrically the hyperplanes  $v^\perp$ and $\bar{v}^\perp$

\begin{lemma}\label{lemma: Rp proj}
$R_{\bar p}$ restricted to $v^\perp$ equals $P_{\bar v}\circ P_{N(\bar p)}^v$, it is injective and $R_{\bar p}(v^\perp)=\bar v^\perp$.
\end{lemma}

\begin{proof}
For $u\in v^\perp$ we have
\begin{align*}
P_{\bar v}\circ P_{N(\bar p)}^v u &=  P_{N(\bar p)}^v u - \langle  P_{N(\bar p)}^v u, \bar v\rangle \bar{v}\\
&=P_{N(\bar p)}^v u - \langle  P_{N(\bar p)}^v u,  v\rangle \bar{v}\\
&=u-\frac{\langle u, N(\bar p)\rangle}{\langle v,N(\bar p)\rangle}v - \langle  u-\frac{\langle u, N(\bar p)\rangle}{\langle v,N(\bar p)\rangle}v,  v\rangle \bar{v}\\
&= u-\langle u, v\rangle \bar{v}+ \frac{\langle u, N(\bar p)\rangle}{\langle v,N(\bar p)\rangle}(\bar{v}-v)\\
&= u-2\frac{\langle u, N(\bar p)\rangle}{\langle v,N(\bar p)\rangle}\langle v,N(\bar p)\rangle N(\bar p)\\
&=u-2\langle u, N(\bar p)\rangle N(\bar p)\\
&=R_{\bar p}u.
\end{align*}
\end{proof}

\subsubsection{Jacobi coordinates}

In neighbourhoods $B$ of $(p,v)$ and $\bar B$ of $(\bar{p},\bar{v})$ in $M$ consider the respective changes of coordinates 
\begin{eqnarray*}
&\Psi \colon B\to (p+v^\perp)\times V_p \\
&\Psi(q,u)  = (p+ P_{v}(q-p),u)
\end{eqnarray*}
and
\begin{eqnarray*}
&\bar \Psi\colon \bar B\to (\bar p+\bar v^\perp)\times V_{\bar p} \\
&\bar \Psi(q,u) = (\bar p+ P_{\bar v}(q-\bar p),u).
\end{eqnarray*}
Notice that $\Psi(p,v)=(p,v)$ and
$$
D\Psi(p,v)= (P_{v},\id)\colon N(p)^\perp\times v^\perp\to v^\perp \times v^\perp.
$$

Since $P_{v}\circ P_{\eta}^v= \id$ on $v^\perp$,  
$$
D\Psi^{-1}(p,v)=(P_{N(p)}^v,\id) \colon  v^\perp \times v^\perp\to N(p)^\perp\times v^\perp.
$$

For a sufficiently small neighborhood $B$ of $(p,v)$ in $M$, let 
$$
\tilde{f}=\bar{\Psi}\circ f\circ \Psi^{-1}
\colon \Psi(B)\subset  ( p+ v^\perp)\times V_{p} \to (\bar p+\bar v^\perp)\times V_{\bar p}.
$$

Consider a curve 
$$
(p_\varepsilon,v_\varepsilon)
=\left(p+\varepsilon J ,\frac{v+\varepsilon J'}{\|v+\varepsilon J'\|}\right)
\in \Psi(B),
$$ 
where $(J,J')\in v^\perp\times v^\perp$ and $|\varepsilon|$ is sufficiently small.
Clearly we have 
$$
(p_\varepsilon,v_\varepsilon)=(p,v)+\varepsilon (J,J')+\OO(\varepsilon^2).
$$

Now,
\begin{eqnarray*}
\Psi^{-1}(p_\varepsilon,v_\varepsilon)
&=&
(p,v)+\varepsilon D\Psi^{-1}(p,v)(J,J') +\OO(\varepsilon^2) \\
&=&
(p+\varepsilon u , v+\varepsilon w) +\OO(\varepsilon^2)
\end{eqnarray*}
where $u =P_{N(p)}^v J  \in N(p)^\perp$ and $w=J' \in v^\perp$.

Moreover,
$$
(\bar{p}_\varepsilon, \bar{v}_\varepsilon)=f\circ\Psi^{-1}(p_\varepsilon,v_\varepsilon)
$$
is given by
\begin{eqnarray}
\label{eqn: bar p eps}
\bar{p}_\varepsilon 
&=& p + \varepsilon u+\tau(p+\varepsilon u, v+\varepsilon w )(v+\varepsilon w) + \OO(\varepsilon^2)  \\
\label{eqn: bar v eps}
\bar{v}_\varepsilon 
&=& v+\varepsilon w -2 \langle v+\varepsilon w, N(\bar{p}_\varepsilon)\rangle N(\bar{p}_\varepsilon) + \OO(\varepsilon^2).
\end{eqnarray}

\begin{lemma}
$$
\bar{p}_\varepsilon 
=\bar{p}+\varepsilon P_{N(\bar{p})}^v\left( J+\tau(p,v) J'\right)+\OO(\varepsilon^2).
$$
\end{lemma}

\begin{proof}

A convenient decomposition of $\tau$  is obtained by using the fact that tangent vectors at $(p,v)$ are mapped into tangent vectors at $(\bar p,\bar v)$, which are in $N(\bar p)^\perp$. 
That is, $\langle \bar{p}_\varepsilon-\bar p, N(\bar{p})\rangle=\OO(\varepsilon^2)$.
Therefore,
\begin{align*}
\langle \bar{p}_\varepsilon-p, N(\bar{p})\rangle 
&= \langle \bar{p}_\varepsilon-\bar{p}+\bar{p}-p, N(\bar{p})\rangle \\
&=\langle \bar{p}-p, N(\bar{p})\rangle + \OO(\varepsilon^2)\\
&= \tau(p,v)\langle v, N(\bar{p})\rangle + \OO(\varepsilon^2).
\end{align*}
On the other hand, from~\eqref{eqn: bar p eps} one gets
$$
\langle \bar{p}_\varepsilon-p, N(\bar{p})\rangle = 
\langle\varepsilon u+\tau(p+\varepsilon u, v+\varepsilon w )(v+\varepsilon w), N(\bar{p})\rangle 
+ \OO(\varepsilon^2).
$$
So,
\begin{align*}
\tau(p+\varepsilon u,v+\varepsilon w)  &= \frac{\tau(p,v)\langle v, N(\bar{p})\rangle -\varepsilon \langle u, N(\bar{p})\rangle}{\langle v+\varepsilon w, N(\bar{p})\rangle}  +  \OO(\varepsilon^2) \\
&=
\frac{\tau(p,v) -\varepsilon \frac{\langle u, N(\bar{p})\rangle}{\langle v, N(\bar{p})\rangle}}{1+ \varepsilon\frac{\langle w, N(\bar{p})\rangle}{\langle v, N(\bar{p})\rangle}} +  \OO(\varepsilon^2)\\
&=
\tau(p,v)-\varepsilon\left(\tau(p,v)\frac{\langle w, N(\bar{p})\rangle}{\langle v, N(\bar{p})\rangle}+ \frac{\langle u, N(\bar{p})\rangle}{\langle v, N(\bar{p})\rangle} \right)+\OO(\varepsilon^2).
\end{align*}
Finally, ~\eqref{eqn: bar p eps} can be written as
\begin{align*}
\bar{p}_\varepsilon &= p+\tau(p,v)v + \varepsilon\left( u-\frac{\langle u, N(\bar{p})\rangle}{\langle v, N(\bar{p})\rangle}v+\tau(p,v)\left(w-\frac{\langle w, N(\bar{p})\rangle}{\langle v, N(\bar{p})\rangle}v\right)\right) +\OO(\varepsilon^2)\\
&=\bar{p}+\varepsilon\left(P_{N(\bar{p})}^v u+\tau(p,v)P_{N(\bar{p})}^v w\right)+\OO(\varepsilon^2).
\end{align*}
The claim follows from the definitions of $u$ and $w$, as well as ~\eqref{eq: proj eta' proj eta}.
\end{proof}

\begin{lemma}
$$
\bar{v}_\varepsilon 
=\bar{v}+ \varepsilon\left[K R_{\bar{p}} J + \left(I+\tau(p,v)K\right)R_{\bar{p}} J'\right] +\OO(\varepsilon^2)
$$
where $K\colon \bar{v}^\perp\to \bar{v}^\perp$ is the self-adjoint linear map
$$
K = -2\langle \bar{v},N(\bar{p})\rangle (P_{N(\bar{p})}^{\bar{v}})^*L(\bar{p})P_{N(\bar p)}^{\bar v}.
$$
\end{lemma}

\begin{proof}
We first expand $N(\bar{p}_\varepsilon)$ in powers of $\varepsilon$:
\begin{align*}
N(\bar{p}_\varepsilon)
& =N(\bar{p})+\varepsilon DN(\bar{p})\,P_{N(\bar{p})}^v(J+\tau(p,v)J')+\OO(\varepsilon^2).
\end{align*}
Substituting in ~\eqref{eqn: bar v eps} we obtain,
$$
\bar{v}_\varepsilon =\bar{v}+ \varepsilon\left[\hat{K}J + \left(R_{\bar{p}}+\tau(p,v)\hat{K}\right)J'\right] + \OO(\varepsilon^2),
$$
where $\hat K\colon v^\perp\to v^\perp$ is the linear map
$$
\hat K(\zeta)=  2\langle v,L(\bar{p})P_{N(\bar{p})}^v\zeta\rangle N(\bar{p}) +2\langle v,N(\bar{p})\rangle L(\bar{p})P_{N(\bar{p})}^v \zeta
$$
and $L(\bar p)=-DN(\bar{p})$.

By the fact that $L(\bar{p})P_{N(\bar{p})}^v$ maps tangent vectors to tangent vectors, we have
\begin{align*}
\langle v,L(\bar{p})P_{N(\bar{p})}^v\zeta \rangle N(\bar{p}) &= \langle \bar{v},L(\bar{p})P_{N(\bar{p})}^v\zeta \rangle N(\bar{p})\\
&=\langle \bar{v},N(\bar{p})\rangle \frac{\langle \bar{v},L(\bar{p})P_{N(\bar{p})}^v\zeta\rangle}{\langle \bar{v},N(\bar{p})\rangle} N(\bar{p})\\
&=\langle \bar v,N(\bar{p})\rangle\left(L(\bar{p})P_{N(\bar{p})}^v\zeta-P_{\bar{v}}^{N(\bar{p})}L(\bar{p})P_{N(\bar{p})}^v\zeta\right).
\end{align*}
Using~\eqref{eq: bar v-v}, ~\eqref{eq: proj eta v proj v} and Lemma~\ref{lemma: Rp proj},
\begin{align*}
\hat K &= -2\langle \bar{v},N(\bar{p})\rangle (P_{N(\bar{p})}^{\bar{v}})^*L(\bar{p})P_{N(\bar{p})}^v\\
&=-2\langle \bar{v},N(\bar{p})\rangle (P_{N(\bar{p})}^{\bar{v}})^*L(\bar{p})P_{N(\bar p)}^{\bar v} R_{\bar p} \\
&=KR_{\bar p}.
\end{align*}
\end{proof}

Denoting $(\tilde{p}_\varepsilon,\tilde{v}_\varepsilon)=\tilde{f}(p_\varepsilon,v_\varepsilon)$, by ~\eqref{eq: proj eta' proj eta} and Lemma~\ref{lemma: Rp proj} we conclude that
\begin{align}\label{eq: tilde f}
\tilde{p}_\varepsilon 
&=
\bar{p} +  \varepsilon P_{\bar{v}}\circ P_{N(\bar{p})}^v\left(J + \tau(p,v) J'\right) + \OO(\varepsilon^2)\\
&=
\bar{p} +  \varepsilon \left(R_{\bar{p}}J + \tau(p,v) R_{\bar{p}}J'\right) + \OO(\varepsilon^2),\\
\tilde{v}_\varepsilon 
&=
\bar{v}+ \varepsilon\left[K R_{\bar{p}} J + \left(I+\tau(p,v)K\right)R_{\bar{p}} J'\right] +\OO(\varepsilon^2).
\end{align}

\subsubsection{Derivative of the billiard map}

Recall that
$$
Df(p,v)=D\bar\Psi^{-1}(\bar p,\bar v)\, D\tilde f(p,v)\, D\Psi(p,v).
$$
From the above results we only need to compute the derivative of $\tilde f$.  Given $x=(p,v)\in M$ let $K(x)\colon v^\perp\to v^\perp$ be the self-adjoint linear map
\begin{equation}\label{definition of K}
K(x)w = -2\langle v,N(p)\rangle (P_{N(p)}^{v})^*L(p)P_{N(p)}^{v}w,\quad w\in v^\perp.
\end{equation}

\begin{lemma}\label{lemma derivative billiard map Jacobi coord}
$D\tilde{f}(x)\colon v^\perp\times v^\perp\to \bar{v}^\perp\times \bar{v}^\perp$
is given by
$$
D\tilde{f}(x)\,(J,J')=
\begin{bmatrix}
I&0\\
K(\bar{x})&I
\end{bmatrix}
\begin{bmatrix}
R_{\bar p}&0\\
0& R_{\bar p}
\end{bmatrix}
\begin{bmatrix}
I&\tau(x)I\\
0&I
\end{bmatrix}
\begin{bmatrix}
J\\ J'
\end{bmatrix}.
$$
\end{lemma}

\begin{proof}
It follows from~\eqref{eq: tilde f} that
\begin{eqnarray*}
D\tilde{f}(p,v)(J,J')
&=&
\frac{d}{d\varepsilon}\tilde{f}(p_\varepsilon,v_\varepsilon)|_{\varepsilon=0} \\
&=&
( R_{\bar{p}}J + \tau(p,v)  R_{\bar{p}}J',K R_{\bar{p}} J + \left(I+\tau(p,v)K\right)R_{\bar{p}} J').
\end{eqnarray*}
\end{proof}

\begin{remark}
Denote by $\Omega$ the canonical symplectic form on $\Rr^{2d+2}$. Recall that $M \subset  \Rr^{2d+2}$. It follows from Lemma~\ref{lemma derivative billiard map Jacobi coord}, that the billiard map $f$ is a symplectomorphism with respect to the symplectic form $$\omega = \Omega|_{M}.$$
\end{remark}

\subsection{Generic properties concerning periodic orbits}
\label{sec: generic properties}

Given any billiard $\phi\in \emb 2$, 
we call 
a point $p\in M_\phi$ periodic if its period $m$ is $\geq2$.
All the points at the boundary of $M_\phi$ are the only fixed points, and are not periodic points according to our definition.

A periodic point $p$ is called hyperbolic if the eigenvalues of $Df_\phi^m(p)$ are all outside the unit circle.
It is $q$-elliptic (or simply elliptic) if $Df_\phi^m(p)$ has exactly $2q$ non-real eigenvalues with modulus 1 and $1\leq q\leq d$. 
In case $q=d$ it is called totally elliptic. 
When there are eigenvalues $\pm1$ it is called degenerate.
Finally, if all eigenvalues are $\pm1$ it is called parabolic.  

The number of periodic points is given by the following result.

\begin{theorem}[\cite{farber-2002,farber-tabachnikov-2002,karasev2009}]\label{thm: infinito}
If $\phi\in\emb \infty$, the number of periodic orbits is infinite.
\end{theorem}

Concerning generic billiards, more can be said about their periodic orbits.

\begin{theorem}[Petkov, Stojanov]\label{thm: generic}
For any $r\in\Nn\cup\{\infty\}$, there is a $C^r$-residual set $\RR\subset C_{emb}^{r}(S,\Rr^{d+1})$ such that the billiard map on $\phi\in\RR$ satisfies the following conditions:
\begin{enumerate}
\item
every periodic orbit passes only once through each of its reflection points, and any two different periodic orbits have no common reflection point~\cite{stojanov-ETDS-1987}.
\item
the spectrum of the derivative at any periodic point does not contain roots of unity~\cite{petkov-stojanov-mathZ-1987}.
\item
the number of periodic points with fixed finite period is finite~\cite{petkov-stojanov-AJM-1987,petkov-ETDS-1988}. 
\end{enumerate}
\end{theorem}

\section{Proof of Theorem~\ref{main thm}}\label{sec: proof}

We start by showing the result for the subset of smooth convex bodies that belong to $\emb\infty$.
The set 
$$
\{\phi\in\BB^\infty\colon \phi\text{ has a nontrivial hyperbolic basic set}\}
$$
is $C^2$-open by the strong structural stability of hyperbolic sets (see e.g. ~\cite[Theorem 18.2.1]{Katok}).
It remains to show that it is also $C^2$-dense.

Given $\phi\in \emb \infty$, we denote by $\Per_m(\phi)$ the set of periodic points of the billiard map $f_\phi$ with period $\geq m$, by $E(\phi)$ the subset of $q$-elliptic points, $1\leq q\leq d$, and by $H(\phi)$ the hyperbolic points.

Define the integer
$$
m_d:=4\left(\begin{smallmatrix}2d+3\\4\end{smallmatrix}\right)
$$
and the set of billiards with a $q$-elliptic point with period $\geq m_d$,
$$
\EE:=\{\phi\in \emb\infty\colon 
\Per_{m_d}(\phi)\cap E(\phi)\not=\emptyset\}.
$$
The restriction to large periods comes from the conditions of Theorem~\ref{thm: K-T}, needed to prove Theorem~\ref{thm: ell} below.

The following result implies that any billiard in $\EE$ is $C^\infty$-approximated by another one with a horseshoe, an example of a nontrivial hyperbolic basic set. 
The proof is in section~\ref{sec: ellipt}.

\begin{theorem}\label{thm: ell}
There is a $C^\infty$-residual set $\RR\subset\emb \infty$ such that any $\phi\in\RR\cap\EE$ has a horseshoe.
\end{theorem}

Consider now the set of billiards  for which all its periodic points with large enough period are hyperbolic,
\begin{equation*}\label{estrelafraco}
\HH:=\{\phi\in \emb \infty \colon \Per_{m_d}(\phi) \subset H(\phi) \}
\end{equation*}
and its interior in the $C^2$-topology,
$$
\FF^2=\Int_{C^2}\HH.
$$
We show next that a $C^\infty$-generic billiard that is also in $\FF^2$ has a nontrivial hyperbolic basic set. This is proved in section~\ref{sec: hyp}.
The reason for the restriction to the $C^2$ topology is due to the use of our version for billiards in bodies of the Franks' Lemma in Theorem~\ref{Franks}.

\begin{theorem}\label{thm: hyp}
There is a $C^\infty$-residual set $\RR\subset\emb \infty$ such that 
the closure of $\Per_{m_d}(\phi)$ contains a nontrivial hyperbolic basic set for every $\phi\in \RR\cap \FF^2$.
\end{theorem}

The remaining case, i.e. $\phi$ at the $C^2$-boundary of $\HH$, is reduced by a $ C^2$ perturbation to 
the case inside $\EE$. 
Notice that $\EE\cup\HH$ is $C^\infty$-residual (see Theorem~\ref{thm: generic}).
Thus, the above theorems imply that the set of billiards in $\emb\infty$ with a nontrivial hyperbolic basic set is $C^2$-dense. 
This completes the proof of Theorem~\ref{main thm} restricting to $\emb\infty$.

To deal with billiards in $\XX^\infty$, i.e. on smooth convex bodies, one only needs to notice that $\emb\infty$ is $C^2$ open and dense in $\XX^\infty$.

\section{Perturbing the $k$-jets: Klingenberg-Takens theorem for billiards}
\label{sec: KT}


Let $U$ be a neighbourhood of $0\in\Rr^{m}$ and $f,g\colon U\to\Rr^{m}$ be two smooth maps fixing the origin, i.e. $f(0)=g(0)=0$.  Given $k\in\Nn$,  we say that $f$ and $g$ are $k$-equivalent if they have the same Taylor polynomial of degree $k$ at $0$.  The equivalence class of $f$ under this equivalence relation is the $k$-jet of $f$ at $0$ which we denote by $J^k_0f$ or simply $J^kf$.  We denote by $\JJ^{k}_s(n)$ the set of $k$-jets $J^kf$ of symplectomorphisms $f$ fixing the origin of $\Rr^{2n}$ with the canonical symplectic structure.  Notice that $\JJ^{k}_s(n)$ is a Lie group with the group operation
$$
(J^kf)\cdot (J^kg):=J^k( f\circ g).
$$
A subset $\Sigma\subset\JJ^{k}_s(n)$ is called \textit{invariant} if
$$
\sigma\cdot \Sigma \cdot \sigma^{-1} = \Sigma,\quad\forall\,\sigma\in \JJ^{k}_s(n).
$$

Let $x\in M_\phi$ be a periodic point of period $m\in\Nn$ of $f_\phi$.  Using Darboux coordinates about $x$,  we may assume that the $k$-jet of $f^m_\phi$ at $x$, which we denote by $J^k_x f_\phi^m$,  belongs to $\JJ^{k}_s(d)$.  Clearly,  if $\Sigma$ is invariant, then the property $J^k_x f_\phi^m\in \Sigma$ is independent of the coordinate system.

\begin{theorem}\label{thm: K-T first}
Let $\phi\in\BB^r$, $r\in\{2,3,\ldots,\infty\}$, and $\Sigma$ be an open,  dense and invariant subset of $ \JJ^{k}_s(d)$ with $k\in\Nn$. 
If $x\in M_\phi$ is a periodic point of $f_\phi$ with period $m\geq  4 {2d+k \choose k+1}$ whose periodic orbit $O(x)$ passes only once through each of its reflection points,  then there is a smooth $u\colon S\to\Rr^{d+1}$ with $C^\infty$-norm arbitrary small such that
\begin{enumerate}
\item
$\phi_u:=\phi+u\in\BB^r$, 
\item $O(x)$ is a periodic orbit of $f_{\phi_u}$, 
\item
$J^k_x f^m_{\phi_u}\in \Sigma$.
\end{enumerate}
\end{theorem}

In the spirit of a theorem by Klingenberg and Takens~\cite{KT-1972}, we then show the following version for multidimensional billiards.

\begin{theorem}\label{thm: K-T}
For every open, dense and invariant $\Sigma\subset \JJ^{k}_s(d)$, $k\in\Nn$, there is a $C^\infty$-residual set $\RR=\RR(\Sigma)\subset \emb\infty$ such that for every $\phi\in\RR$ and any periodic point $x$ of $f_\phi$ of period $m\geq  4 {2d+k \choose k+1}$ we have $J^k_x f_\phi^m\in \Sigma$.
\end{theorem}

We prove Theorems~\ref{thm: K-T first} and ~\ref{thm: K-T} at the end of this section.

\subsection{Perturbing the $k$-jet of the billiard map}

Let $\phi\in\BB^r$, $r\in\{2,3,\ldots,\infty\}$.  Consider an orbit segment $\gamma=\{x_0,x_1,x_2\}\subset M_\phi$ of $f_\phi$ where $\pi_1(\gamma)$ consists of exactly three points in $\Gamma_\phi$. On $p_1=\pi_1(x_1)$ we perform a perturbation 
\begin{equation}\label{eq:perturbation u}
\phi_\varepsilon = (\id+\varepsilon u N_\phi)\circ \phi 
\end{equation}
where $u\colon \Rr^{d+1}\to\Rr$ is a compactly supported $C^\infty$ function with support contained in a neighbourhood of $p_1$ not intersecting $\{p_0,p_2\}$,  $u(p_1)=0$ and $\nabla u(p_1)=0$.  

\begin{lemma}\label{lem:Nepsilon} If $|\varepsilon|$ is sufficiently small, then $\phi_\varepsilon\in \BB^{r}$. Moreover, 
$$
N_{\phi_\varepsilon}(p+\varepsilon u(p)N_\phi(p))=N_\phi(p)-\varepsilon P_{N_\phi(p)}(\nabla u(p)) + \OO(\varepsilon^2).
$$
\end{lemma}

\begin{proof}
That $\phi_\varepsilon\in \BB^r$ for $|\varepsilon|$ small follows the same lines of the proof of Lemma~\ref{lem:perturbing body}.  Given $p\in \Gamma_\phi$,  let $p_\varepsilon = p+\varepsilon u(p)N(p)$ where $N=N_\phi$.  Notice that $p_\varepsilon\in\Gamma_{\phi_\varepsilon}$.  Denote by $\{t_j^\varepsilon(p_\varepsilon)\}$ the basis of $T_{p_\varepsilon}\Gamma_{\phi_\varepsilon}$ as defined in \eqref{eq: tg vectors}. Notice that $t_j(p)= t_j^0(p)$ and 
$$
t_j^\varepsilon(p_\varepsilon) =t_j(p) + \varepsilon (N(p)\nabla u(p)+ u(p)DN(p))t_j(p).
$$
Since $\langle N_{\phi_\varepsilon}(p_\varepsilon), N_{\phi_\varepsilon}(p_\varepsilon)\rangle =1$, we conclude that $w(p):=\left.\frac{d}{d\varepsilon}\right|_{\varepsilon=0} N_{\phi_\varepsilon}(p_\varepsilon) \in T_p\Gamma$ and $N_{\phi_\varepsilon}(p_\varepsilon)=N_\phi(p)+\varepsilon w(p) + \OO(\varepsilon^2)$.  Now we determine $w$.  Taking into account that $\langle N_{\phi_\varepsilon}(p_\varepsilon),   t_j^\varepsilon(p_\varepsilon)\rangle=0$ we get
\begin{align*}
0&=\left.\frac{d}{d\varepsilon}\right|_{\varepsilon=0} \langle N_{\phi_\varepsilon}(p_\varepsilon),   t_j^\varepsilon(p_\varepsilon)\rangle \\&= \langle w(p),   t_j(p)\rangle + \langle N(p),   (N(p)\nabla u(p)+ u(p)DN(p))t_j(p)\rangle\\
&=\langle w(p),   t_j(p)\rangle +\langle \nabla u(p),t_j(p)\rangle +u(p) \langle N(p),DN(p)t_j(p)\rangle\\
&=\langle w(p),   t_j(p)\rangle +\langle \nabla u(p),t_j(p)\rangle
\end{align*}
where the last term vanishes because $DN(p)t_j(p)\in T_p\Gamma$.  Hence, 
$$
\langle w(p),   v\rangle = \langle -\nabla u(p),v\rangle,\quad \forall\,v\in T_p\Gamma.
$$ 
This implies that
$$
w(p)=-\nabla u(p)+\langle \nabla u(p),N(p)\rangle N(p) = -P_{N(p)}(\nabla u(p)).
$$
\end{proof}

Denote by $B_\delta(x_0)\subset M_\phi$ an open ball around $x_0$ of radius $\delta>0$.  There are $0<\delta_0<\delta_1$ such that for every $|\varepsilon|$ sufficiently small,  the map $F_\varepsilon\colon B_{\delta_0}(x_0)\to B_{\delta_1}(x_0)$ defined by 
\begin{equation}\label{def:Rvarepsilon}
F_\varepsilon (x)=  f_\phi^{-2}\circ  f_{\phi_\varepsilon}^2(x),
\end{equation}
is a diffeomorphism onto its image.  Notice that $F_0=\id$ and $F_\varepsilon(x_0)=x_0$.  Let $V = f_{\phi}(B_{\delta_0}(x_0))\subset M_\phi$.

\begin{proposition}\label{prop:chi}
$ F_\varepsilon = \id + \varepsilon f_\phi^*\chi+\OO(\varepsilon^2) $ where $\chi\colon V\to TM_\phi$ is the vector field $\chi(p,v)=(\chi_1(p,v),\chi_2(p,v))$ given by
\begin{align*}
\chi_1(p,v)&=-\frac{2u(p)}{\langle v,N_\phi(p)\rangle}P_{N_\phi(p)}v,\\
\chi_2(p,v)&=2u(p)P_v^{N_\phi(p)}L(p)P_{N_{\phi(p)}}v- 2\langle v,N_\phi(p)\rangle P_v^{N_\phi(p)}P_{N_\phi(p)}\nabla u(p).
\end{align*}
\end{proposition}

\begin{proof}
Given $(p,v)\in V$ let $\bar{v}$ be the reflection of $v$ about the hyperplane perpendicular to $N_\phi(p)$ and let $\ell(p,\bar{v})$ be the line passing through $p$ in the direction of $\bar{v}$.  There is a unique $\bar{p}\in \Gamma_{\phi_\varepsilon}$ which is the first, i.e., closest to $p$,  intersection point of $\Gamma_{\phi_\varepsilon}$ with the line $\ell(p,\bar v)$.  Notice that $\bar{p}$ may be equal to $p$ which will certainly be the case whenever $p\in \Gamma_{\phi}\cap \Gamma_{\phi_\varepsilon}$.  
We define the map $g_\varepsilon\colon V\to M_{\phi_\varepsilon}$ as $(p,v)\mapsto (\bar{p},-\bar{v})$. Notice that $(\bar{p},-\bar{v})\in M_{\phi_\varepsilon}$. Indeed,  by continuity we have $\langle\bar{v},N_{\phi_\varepsilon}(\bar{p})\rangle <0$, since $N_{\phi_\varepsilon}(\bar{p})$ is $\varepsilon$-close to $N_{\phi}(p)$ and $\langle\bar{v},N_{\phi}(p)\rangle <0$. 
Moreover,  $g_\varepsilon$ is a local diffeomorphism at $x_1$.  
Similarly,  let $V_\varepsilon = g_\varepsilon(V)\subset M_{\phi_\varepsilon}$ and given $(p,v)\in V_\varepsilon$ let $\bar{p}$ be the first intersection point of $\Gamma_{\phi}$ with the line $\ell(p,\bar{v})$ where now $\bar{v}$ is the reflection of $v$ about the hyperplane perpendicular to $N_{\phi_\varepsilon}(p)$.  The map $h_\varepsilon\colon V_\varepsilon\to M_\phi$ is defined by $(p,v)\mapsto (\bar{p},-\bar{v})$. As before,  $(\bar{p},-\bar{v})\in M_\phi$ and $h_\varepsilon$ is a local diffeomorphism at $x_1$.  Finally, define $G_\varepsilon \colon V\to M_\phi$ by $G_\varepsilon := h_\varepsilon\circ g_\varepsilon$.   
It is not difficult to see that $G_0=\id$, $G_\varepsilon(x_1)=x_1$ and 
$$F_\varepsilon= f_\phi^{-1}\circ  G_\varepsilon\circ f_\phi. $$
Now we expand $G_\varepsilon$ in leading order of $\varepsilon$.  Starting with $g_\varepsilon$,  let $(\bar{p}_\varepsilon,-\bar{v}) = g_\varepsilon(p,v)$. Clearly, 
$$
\bar{p}_\varepsilon = p + \tau_\varepsilon \bar{v}\quad\text{and}\quad \bar{v}= R_{N_\phi(p)} v
$$
where $\tau_\varepsilon = \langle \bar{p}_\varepsilon-p,\bar{v}\rangle$.  Since $\bar{p}_\varepsilon\to p$ as $\varepsilon\to 0$, we have 
$$
\tau_\varepsilon = \varepsilon\bar{\tau} + \OO(\varepsilon^2)
$$ 
for some function $\bar{\tau}=\bar{\tau}(p,\bar{v})$ that we now determine. As $\bar{p}_\varepsilon\in \Gamma_{\phi_\varepsilon}$, there is $s_\varepsilon\in S$ such that $s_\varepsilon\to \phi^{-1}(p)$ as $\varepsilon\to 0$ and
$$
\bar{p}_\varepsilon = \phi(s_\varepsilon)+\varepsilon u \circ\phi(s_\varepsilon) N_\phi\circ\phi(s_\varepsilon).
$$
Therefore,
$$
\phi(s_\varepsilon)-p = \varepsilon(\bar{\tau}\bar{v}-u(p)N_\phi(p))+\OO(\varepsilon^2),
$$
which implies that $\bar{\tau}\bar{v}-u(p)N_\phi(p)\in T_p\Gamma_\phi$.  Equivalently,  $$\langle \bar{\tau}\bar{v}-u(p)N_\phi(p), N_\phi(p)\rangle = 0,$$ which gives
$$
\bar{\tau}=\frac{u(p)}{\langle \bar{v},N_\phi(p)\rangle}.
$$
Putting all together,
\begin{align*}
\bar{p}_\varepsilon &= p + \varepsilon\frac{u(p)}{\langle \bar{v},N_\phi(p)\rangle} \bar{v} + \OO(\varepsilon^2)\\
&=p - \varepsilon\frac{u(p)}{\langle v,N_\phi(p)\rangle}R_{N_\phi(p)}v + \OO(\varepsilon^2).
\end{align*}
Next, we expand $(\tilde{p}_\varepsilon,\tilde{v}_\varepsilon)=h_\varepsilon(\bar{p}_\varepsilon,-\bar{v})$ in leading order of $\varepsilon$. First,  notice that defining $w=w(p,v):=\bar{\tau}\bar{v}-u(p)N_\phi(p)$ we have
\begin{align*}
N_{\phi_\varepsilon}(\bar{p}_\varepsilon)&=N_{\phi_\varepsilon}(p+\varepsilon u(p)N_\phi(p) + \varepsilon w + \OO(\varepsilon^2))\\
&=N_{\phi_\varepsilon}(p+\varepsilon u(p)N_\phi(p))+ \varepsilon DN_
\phi(p)w+\OO(\varepsilon^2)\\
&=N_{\phi}(p)+ \varepsilon( DN_
\phi(p)w-P_{N_\phi(p)}\nabla u(p) )+\OO(\varepsilon^2),
\end{align*}
by Lemma~\ref{lem:Nepsilon}.  Let
$$
\Theta(p,v):= DN_
\phi(p)w-P_{N_\phi(p)}\nabla u(p).
$$
Notice that $\Theta(p,v)\in T_p\Gamma$.
 Expanding $\tilde{v}_\varepsilon = R_{N_{\phi_\varepsilon}(\bar{p}_\varepsilon)}\bar{v}$ in powers of $\varepsilon$ we get
\begin{align*}
\tilde{v}_\varepsilon &= \bar{v}-2\langle \bar{v},N_{\phi_\varepsilon}(\bar{p}_\varepsilon)\rangle N_{\phi_\varepsilon}(\bar{p}_\varepsilon)\\
&=R_{N_\phi(p)}\bar{v} -2 \varepsilon(\langle \bar{v}, \Theta(p,v) \rangle N_\phi(p) + \langle \bar{v},  N_\phi(p)\rangle\Theta(p,v)) + \OO(\varepsilon^2).
\end{align*}
Taking into account that $\bar{v}=R_{N_\phi(p)}v$ and $$\langle \bar{v}, \Theta(p,v) \rangle N_\phi(p)  = \langle v,N_\phi(p)\rangle (\Theta(p,v)-P_v^{N_\phi(p)}\Theta(p,v)),$$ we get
$$
\tilde{v}_\varepsilon = v +2\varepsilon \langle v,N_\phi(p)\rangle P_v^{N_\phi(p)}\Theta(p,v) + \OO(\varepsilon^2).
$$
Now, by a direct computation we have
$$
w(p,v)=-\frac{u(p)}{\langle v, N_\phi(p)\rangle} P_{N_\phi(p)} v,
$$
and
$$
P_v^{N_\phi(p)}\Theta(p,v)=-P_v^{N_\phi(p)}P_{N_\phi(p)}\nabla u(p)+\frac{u(p)}{\langle v, N_\phi(p)\rangle}P_v^{N_\phi(p)}L(p)P_{N_{\phi(p)}}v.
$$

Putting all together,
$$
\tilde{v}_\varepsilon = v +2 \varepsilon \left(u(p)P_v^{N_\phi(p)}L(p)P_{N_{\phi(p)}}v- \langle v,N_\phi(p)\rangle P_v^{N_\phi(p)}\nabla u(p)\right)+\OO(\varepsilon^2).
$$
Finally,  we expand $\tilde{p}_\varepsilon$ in leading order of $\varepsilon$. First,  notice that $\langle \tilde{p}_\varepsilon-\bar{p}_\varepsilon, \tilde{v}_\varepsilon\rangle = \tilde{\tau}\varepsilon + \OO(\varepsilon^2)$   where $ \tilde{\tau} = \tilde{\tau}(p,v)$ is a function to be determined.  Hence,
\begin{align*}
\tilde{p}_\varepsilon &= \bar{p}_\varepsilon + \varepsilon \tilde{\tau}v + \OO(\varepsilon^2)\\
&=  \phi(s_\varepsilon)+\varepsilon u \circ\phi(s_\varepsilon) N_\phi\circ\phi(s_\varepsilon) + \varepsilon\tilde{\tau}v + \OO(\varepsilon^2)\\
&=\phi(s_\varepsilon)+\varepsilon (u (p) N_\phi(p) + \tilde{\tau}v) + \OO(\varepsilon^2)\\
&=p+\varepsilon(u (p) N_\phi(p) + \tilde{\tau}v + w)\varepsilon+\OO(\varepsilon^2).
\end{align*}
Since  $\tilde{p}_\varepsilon\in\Gamma$ and converges to $p$ as $\varepsilon\to 0$, we conclude that $u (p) N_\phi(p) + \tilde{\tau}v + w\in T_p \Gamma$, thus
$$
\tilde{\tau}=-\frac{u(p)}{\langle v,N_\phi(p)\rangle}.
$$
Therefore,
\begin{align*}
\tilde{p}_\varepsilon &= \bar{p}_\varepsilon -\varepsilon\frac{u(p)}{\langle v,N_\phi(p)\rangle}v + \OO(\varepsilon^2)\\
&=p-\varepsilon \frac{u(p)}{\langle v,N_\phi(p)\rangle}(v+R_{N_\phi(p)}v)+\OO(\varepsilon^2)\\
&=p-\varepsilon \frac{2u(p)}{\langle v,N_\phi(p)\rangle}P_{N_\phi(p)}v+\OO(\varepsilon^2).
\end{align*}
\end{proof}

\begin{proposition}\label{prop:XH}
$$
\left.\frac{d}{d\varepsilon}\right|_{\varepsilon=0} F_\varepsilon = X_H
$$
where $X_H$ is the Hamiltonian vector field of the Hamiltonian 
\begin{equation}\label{eq:G}
H=h\circ f_\phi,\quad h(p,v)=2u(p)\langle v,N_\phi(p)\rangle
\end{equation}
with respect to the symplectic form $\omega$.
\end{proposition}

\begin{proof}
Let $\chi$ be the vector field of Proposition~\ref{prop:chi}. Given $(\eta,\gamma)\in T_{(p,v)} M_\phi$ we find that
$$
\omega(\chi(p,v),(\eta,\gamma))=-2 u(p)\frac{\langle P_{N_\phi(p)}v,\gamma\rangle}{\langle v,N_\phi(p)\rangle} + 2  \langle Z,\eta\rangle
$$
where 
$$
Z:=u(p) P_v^{N_\phi(p)}DN_\phi(p)P_{N_\phi(p)}v +  \langle v, N_\phi(p)\rangle P_v^{N_\phi(p)} P_{N_\phi(p)}\nabla u(p).
$$
Taking into account that $\eta\in T_p\Gamma$ we get
\begin{align*}
\langle Z,\eta\rangle &= u(p)\langle DN_\phi(p)P_{N_\phi(p)}v,\eta\rangle + \langle v, N_\phi(p)\rangle\langle P_{N_\phi(p)}\nabla u(p),\eta\rangle\\
&=u(p)\langle DN_\phi(p)\eta,P_{N_\phi(p)}v\rangle + \langle v, N_\phi(p)\rangle \langle \nabla u(p),\eta\rangle\\
&=u(p)\langle DN_\phi(p)\eta,v\rangle + \langle v, N_\phi(p)\rangle \langle \nabla u(p),\eta\rangle.
\end{align*}
Moreover, since $\gamma\in v^\perp$, we have
$$
\langle P_{N_\phi(p)}v,\gamma\rangle = -\langle v, N_\phi(p)\rangle\langle \gamma, N_\phi(p)\rangle.
$$
Now, taking the derivative of $h(p,v)=2u(p)\langle v,N_\phi(p)\rangle$ we see that
\begin{align*}
dh(p,v)(\eta,\gamma)= 2u(p)&\left(\langle v, D N_\phi(p)\eta\rangle
+\langle \gamma, N_\phi(p)\rangle\right)\\&+ \langle v, N_\phi(p)\rangle\langle \nabla u(p),\eta\rangle.
\end{align*}
Thus, $\omega(\chi,\cdot) = dh$, which shows that $\chi$ is Hamiltonian with respect to $\omega$. By Proposition~\ref{prop:chi},  $X:=\left.\frac{d}{d\varepsilon}\right|_{\varepsilon=0} F_\varepsilon = f_\phi^* \chi$,   hence $X$ is the Hamiltonian vector field of $h\circ f_\phi$ with respect to $\omega$.
\end{proof}

Given $k\in\Nn$, denote by $\Rr_{k+1}[y,z]$ the vector space of homogeneous polynomials of degree $k+1$ with real coefficients in the variables $y=(y_1,\ldots,y_{d})$ and $z=(z_1,\ldots,z_{d})$.  Let $$\ell=\ell(d,k):=\dim  \Rr_{k+1}[y,z].$$  It is clear that 
\begin{equation}\label{eq:ell}
\ell(d,k) = {2d+k \choose k+1}.
\end{equation}   
For $G\in  \Rr_{k+1}[y,z]$ given by
\begin{equation}\label{eq:F}
G(y,z)=y_1^{k+1},
\end{equation}
we define
$$
\GG_k:=\{(A_1,\ldots,A_{\ell})\in \Sp(\Rr^{2d})^\ell\colon \text{span}\{G\circ A_i\}_{i=1}^\ell = \Rr_{k+1}[y,z]\},
$$
where $\Sp(\Rr^{2d})$ stands for the symplectic linear group on $\Rr^{2d}$.

\begin{proposition}[\cite{carballo-miranda-2013}] For each $k\in\Nn$,  the subset $\GG_k$ is open and dense in $\Sp(\Rr^{2d})^\ell$.
\end{proposition}

Let $\{x_0,x_1,\ldots,x_{n}\}\subset M_\phi$ be an orbit segment of $f_\phi$ with $n\geq2$.  Let $(V_i,\psi_i)$ be a Darboux chart around $x_i=(p_i,v_i)$, where $\psi_i\colon V_i\to\Rr^{2d}$ is such that $\psi_i(x_i)=0$.  On $V_i$ we have the embedding of $\Gamma_\phi$ in $M_\phi$ defined by $\iota\colon \pi_1(V_i)\to M_\phi$,  $p\mapsto (p,v_i)$.  As $\iota(\pi_1(V_i))$ is Lagrangian, we may assume that the local coordinates $(y,z)=(y_1,\ldots,y_{d},z_1,\ldots,z_{d})$ given by the Darboux chart  $(V_i,\psi_i)$ satisfy 
\begin{equation}\label{eq:darboux}
\iota(\pi_1(V_i))\cap V_i = \{z_1=0,\ldots,z_{d}=0\}.
\end{equation}
\begin{defn}
We say that $f_\phi$ is \textit{$k$-general} along $\{x_0,x_1,\ldots, x_n\}$ if there exist positive integers $0<n_1<n_2<\cdots <n_\ell< n$ such that $(A_1,\ldots, A_\ell)\in \GG_k$ where $A_i=D(\psi_{n_i}\circ f_\phi^{n_i}\circ \psi_0^{-1})(0)$. 
\end{defn}

In the following we suppose that the orbit segment $\{x_0,x_1,\ldots,x_{n}\}$ passes only once through each of its reflection points, i.e.,  $\pi_1(x_i)\neq \pi_1(x_j)$ whenever $i\neq j$.  As in  \eqref{eq:perturbation u},  for each $i\in\{1,\ldots,n-1\}$, we perturb the body $\Gamma_\phi$ in a neighborhood of $p_i$,
$$
\phi_{\bepsilon} = (\id+(\varepsilon_1 u_1+\ldots+\varepsilon_{n-1}u_{n-1}) N_\phi)\circ \phi,
$$
where $\bepsilon=(\varepsilon_1,\ldots,\varepsilon_{n-1})\in\Rr^{n-1}$ and $u_i\colon\Rr^{d+1}\to\Rr$  is a $C^\infty$ function with compact support $K_i$ containing a neighborhood of $p_i$ such that $K_i\cap K_j=\emptyset$ whenever $i\neq j$.  According to \eqref{eq:darboux},  we also choose $u_i$ so that 
$$
J_0^{k+1} (u_i\circ \pi_1\circ \psi_i^{-1})(y,z) = G(y,z),
$$
where $G$ is defined in \eqref{eq:F}. 
This choice of $u_i$ can always be achieved  using appropriate $C^\infty$ bump functions.

By Lemma~\ref{lem:Nepsilon},  $\phi_{\bepsilon} \in \BB^r $ for every $\bepsilon\in B_\delta(0):=\{\bepsilon\in\Rr^{n-1}\colon \|\bepsilon\|<\delta\}$. In a neighborhood $V\subset V_0$ of $x_0$ we define the map, 
$$
F_{\bepsilon}^{(n)} = f_\phi^{-n} \circ f_{\phi_{\bepsilon}}^{n}.
$$ 
Notice that $F^{(2)}_{\bepsilon}=F_{\varepsilon_1}$ as defined in \eqref{def:Rvarepsilon}.  As in the case of $n=2$,  the map $F_{\bepsilon}^{(n)}$ is a local diffeomorphism at $x_0$.  
Next, we define the map $\SSS\colon B_\delta(0) \to \ker(\pi_k)\subset \JJ_s^k(d)$ by
$$
\SSS(\bepsilon)= J^k_0 (\psi_0\circ F_{\bepsilon}^{(n)}\circ \psi_0^{-1}),
$$
where $\pi_k\colon \JJ_s^k(d)\to \JJ_s^{k-1}(d)$ is the canonical projection.

Notice that $\SSS(0)=\id\in\ker(\pi_k)$. 
\begin{theorem}\label{thm:babyKT}
If $f_\phi$ is $k$-general along $\{x_0,x_1,\ldots, x_n\}$, then
 $\SSS$ is a submersion at $0$.
\end{theorem}

\begin{proof}
Let $\bepsilon_{i}=\varepsilon_i e_i\in\Rr^{n-1}$ where $e_i$ is the vector with value one in the $i-th$ entry and zero elsewhere.  Clearly,
$$
F_{\bepsilon_{i}}^{(n)}= f^{-i+1}\circ F_{\varepsilon_i}\circ f^{i-1}
$$
where $F_{\varepsilon_i}$ is the map defined in \eqref{def:Rvarepsilon}. By Proposition~\ref{prop:XH},  $F_{\varepsilon_i}=\id + \varepsilon_i X_{\tilde{H}_i}+O(\varepsilon_i^2)$ where $\tilde{H}_i$ is the  Hamiltonian $\tilde{H}_i=h_i\circ f_\phi$ with $h_i$ given in \eqref{eq:G}, i.e., $h_i(p,v)=2u_i(p)\langle v,N_\phi(p)\rangle$.  Therefore,
$$
\frac{\partial \SSS}{\partial \varepsilon_i}(0) = J^k_0 X_i,
$$
where $X_i$ is the Hamiltonian vector field of the Hamiltonian $H_i=h_i\circ f_\phi^i\circ \psi_0^{-1}$ with the respect to the standard symplectic form $dy\wedge dz$ in $\Rr^{2d}$.  Notice that 
\begin{align*}
J^{k+1}_0 H_i &= J^{k+1}_0 (h_i\circ f_\phi^i\circ \psi_0^{-1}) \\
&= J^{k+1}_0 (h_i\circ\psi_i^{-1})\circ(\psi_i\circ f_\phi^i\circ \psi_0^{-1})\\
&=2\langle v_i,N_\phi(p_i)\rangle G\circ D(\psi_i\circ f^i_\phi\circ \psi_0^{-1})(0)
\end{align*}
where $G$ is the polynomial given in \eqref{eq:F}.
Since $f_\phi$ is $k$-general along $\{x_0,x_1,\ldots, x_n\}$,  there are positive integers $0<n_1<n_2<\cdots<n_\ell< n$ such that $\{J^{k+1}_0 H_{n_j}\}_{j=1}^\ell$ spans the vector space $\Rr_{k+1}[y,z]$. Thus,  $\{J^k_0 X_{n_j}\}_{j=0}^\ell$ spans the tangent space of the Lie subgroup $\ker(\pi_k)$ at $\id$, which proves that $\SSS$ is a submersion at $0$.
\end{proof}

\subsection{Proof of Theorem~\ref{thm: K-T first}}

Recall $\ell$ from \eqref{eq:ell}.  Let $k\in\Nn$ and $\Sigma$ be an open,  dense and invariant subset of $ \JJ^{k}_s(d)$ and $x\in M_\phi$ be a periodic point of $f_\phi$ with period $m\geq  4 \ell$ whose periodic orbit $O(x)$ passes only once through each of its reflection points.  Splitting the orbit $O(x)$ in blocks of length 4, we can use Lemma~\ref{lem:franks} and Lemma~\ref{lem:Franks generic} together with Lemma~\ref{lem:perturbing body} to perturb the derivatives $\{Df^{4i}_{\phi}(x)\}_{i=1}^\ell$ by an arbitrary $C^\infty$- small smooth perturbation $\phi_0=\phi+u_0$ so that $O(x)$ is still a periodic orbit of $f_{\phi_0}$ and $f_{\phi_0}$ is $j$-general along $O(x)$ for every $j=1,\ldots, k$.  Hence,  by applying $k$ times Theorem~\ref{thm:babyKT}, there is an arbitrary $C^\infty$- small smooth perturbation $\phi_1=\phi_0+u_1$ such that $J^k_x (f_{\phi_0}^{-m}\circ f_{\phi_{1}}^m)\in f^{-m}_{\phi_0}\Sigma$.  This shows that $J_x^k f_{\phi_1}^m \in \Sigma$.
\qed

\subsection{Proof of Theorem~\ref{thm: K-T}}

Recall $\ell$ from \eqref{eq:ell}.  Let $\Sigma$ be an open, dense and invariant subset of $ \JJ^{k}_s(d)$, $k\in\Nn$.   Given $m\geq4\ell$,  denote by $\YY_m$ the set of $\phi\in\emb \infty$ such that $f_\phi$ has only a finite number of periodic orbits of period less than or equal to $m$, all periodic orbits are non-degenerate and each periodic orbit passes only once through each of its reflection points, and any two different periodic orbits have no common reflection point. By Theorem~\ref{thm: generic},  $\YY_m$ is a $C^\infty$-residual subset of $\BB^\infty$.  Now, denote by $\RR_m$ the subset of $\YY_m$ such that for every $\phi\in \RR_m$ and every periodic point $x$ of $f_\phi$ of period $\pi\leq m$, the $k$-jet of $f^\pi_\phi$ at $x$ belongs to $\Sigma$.  Because, for each $\phi\in\RR_m$, the billiard map $f_\phi$ has only a finite number of periodic orbits of period less than $m$ and $\Sigma$ is open,  by continuity of $\phi\mapsto f_\phi$ we conclude that $\RR_m$ is open relative to $\YY_m$.  Moreover, by Theorem~\ref{thm: K-T first}, the set $\RR_m$ is $C^\infty$ dense.  Hence,  $\RR_m$ is a $C^\infty$-residual set.   Taking the intersection, 
$
\RR=\bigcap_{m\geq4\ell}\RR_m
$
we conclude that $\RR$ is also $C^\infty$-residual.
\qed

\section{Proof of Theorem~\ref{thm: ell}}\label{sec: ellipt}

We start by presenting some tools related to elliptic orbits of symplectomorphisms that will be later applied to the billiard maps. We conclude with a perturbation assuring positive topological entropy.

\subsection{Birkhoff normal form}

Let $q\in\Nn$ and $\omega_0$ be the canonical symplectic form on $\Rr^{2q}$.
Given a symplectomorphism $f\colon\Rr^{2q}\to\Rr^{2q}$ so that the origin is a totally elliptic fixed point,
we write the eigenvalues of $Df(0)$ as
$$
e^{\pm 2\pi ia_1},\dots, e^{\pm 2\pi ia_q}.
$$
Moreover, the fixed point is called 4-elementary if 
$$
\sum_{j=1}^qa_j\nu_j \not\in\Zz,
$$
for every $\nu_1,\dots,\nu_q\in\Zz$ such that $1\leq  \sum_{j=1}^q|\nu_j| \leq 4$.

One can find coordinates for which $f$ takes a more explicit form, called the Birkhoff normal form (see below).
We will be using the map $\psi\colon\Rr^{q}\times\Rr^q\to\Cc^q$, $\psi(x,y)=x+iy$.

\begin{theorem}[{Birkhoff normal form~\cite[Lemma 3.3.2]{Klingenberg-1978}}]\label{thm: Birkhoff nf}
Let $f$ be a $C^1$-symplectomorphism on $(\Rr^{2q},\omega_0)$ of class $C^3$ at the origin.
If the origin is a 4-elementary totally elliptic fixed point,
then there is a $C^\omega$-symplectomorphism $h$, a $q\times q$ real matrix $\beta$ and a $C^1$-map $R\colon\Cc^q\to\Cc^q$ such that $F\colon\Cc^q\to\Cc^q$, $F:=\psi\circ h\circ f\circ h^{-1}\circ\psi^{-1}$, is given by
$$
F(z)=\Phi(z)z+R(z)
$$
with
$$
\Phi(z)=
\begin{bmatrix}
e^{2\pi i\varphi_1(z)}&&0\\
&\ddots&\\
0&&e^{2\pi i\varphi_q(z)}
\end{bmatrix}
$$
where $\varphi_j(z)=a_j+\sum_{k=1}^q\beta_{j,k}|z_k|^2$
and $D^\ell R(0)=0$ for $\ell=0,1,2,3$.
\end{theorem}

\subsection{Contreras-Arnaud-Herman theorem}

Consider the exact symplectic manifold $(\Tt^q\times\Rr^q,\omega)$ with
$
\omega = d\lambda
$
and $\lambda=r\,d\theta$ by using the coordinates $(\theta,r)\in\Tt^q\times\Rr^q$.
So, $\omega=dr \wedge d\theta$.
Denote by $\pi_1\colon \Tt^q\times\Rr^q\to\Tt^q$ the canonical projection on the first $q$ components.

A symplectomorphism $f$ is weakly monotonous if
$$
\det (\partial_2 \pi_1f(\theta,r)) \not=0
$$ 
(when $q=1$ it is usually called a twist map).

A completely integrable symplectomorphism is defined to be of the form 
$$
g(\theta,r)=(\theta+\tau(r) \bmod \Zz^q,r)
$$
for a given $\tau\in C^1(\Rr^q,\Rr^q)$ with $\tau(0)=0$.
It is an exact symplectomorphism.
In addition, it is weakly monotonous if $\det(D\tau(\theta,r))\not=0$.
Notice that in this case any symplectomorphism $C^1$-close to $g$ is also weakly monotonous.

The following theorem is proved in~\cite[Theorem 4.1]{contreras-am-2010} by Contreras using previous results of Arnaud and Herman~\cite{Arnaud1992}.

\begin{theorem}[Contreras]\label{th:contreras}
If $f\colon\Tt^q\times\Rr^q\to\Tt^q\times\Rr^q$ is a weakly monotonous exact $C^4$-symplectomorphism without degenerate periodic orbits and $C^1$-close to a completely integrable symplectomorphism, then $f$ has  a 1-elliptic periodic point near $\Tt^q\times\{0\}$.
\end{theorem}

\subsection{Periodic orbits of a generic billiard}

We  show here that generically, periodic orbits are hyperbolic or a specific kind of $q$-ellipticity holds. We make use of the version of the Klingenberg-Takens theorem for billiard maps given by Theorem~\ref{thm: K-T}.

Let $\phi\in \emb r$, $r\in \{2,3,\dots,\infty\}$, and $x$ be a $q$-elliptic periodic point with period $m$.
The map $f_\phi^m$
restricted to the center manifold $W^c(x)$ in a small enough neighbourhood of $x$, can be written in appropriate coordinates as a $C^{r-1}$-diffeomorphism $f\colon\Rr^{2q}\to\Rr^{2q}$ preserving the canonical symplectic form $\omega_0$ and fixing the origin.
If the origin is 4-elementary, then we say that $x$ is a 4-elementary $q$-elliptic periodic point of $\phi$.

Moreover, we say that $x$ is weakly monotonous if there is a Birkhoff normal form (Theorem ~\ref{thm: Birkhoff nf}) satisfying $\det(\beta)\not=0$.
Notice that this condition is invariant under conjugation by symplectomorphisms.
In addition, observe that the 3-jet of the new map $F$ at the origin is solely determined by $\Phi$.

\begin{proposition}\label{prop: generic 4-el wm}
There is a $C^\infty$-residual set $\RR_1\subset \emb \infty$ such that for any $\phi\in\RR_1$ any periodic point of period  $\geq 4\left(\begin{smallmatrix}2d+3 \\ 4\end{smallmatrix}\right)$ is either hyperbolic or 4-elementary weakly monotonous $q$-elliptic for some  $1\leq q\leq d$.
\end{proposition}

\begin{proof}
Let $\Sigma$ be the subset of 3-jets in $\JJ^3_s(d)$ for which the origin is either hyperbolic or 4-elementary weakly monotonous $q$-elliptic for some  $1\leq q\leq d$.
Notice that $\Sigma$ is open, dense and invariant. 
So, by Theorem~\ref{thm: K-T}, there is a residual set of billiards in $\BB^\infty$ whose 3-jets are in $\Sigma$. 
\end{proof}

\subsection{Reduction to 1-elliptic}
\label{sec: reduction to 1-ell}

We now find conditions for the existence of a 1-elliptic periodic point nearby a $q$-elliptic one.

Consider $\phi\in \emb2$ and a 4-elementary weakly monotonous $q$-elliptic $m$-periodic point.
We look again at $f_\phi^m$ restricted to the center manifold of the period point. This defines a Birkhoff normal form $F$ as in Theorem~\ref{thm: Birkhoff nf}, for the coordinates $z=x+iy\in\Cc^q$, and $\widetilde F=\psi^{-1}\circ F\circ \psi$ for the coordinates $(x,y)$.

Let $\varepsilon>0$ and
$$
Q_\varepsilon\colon \Rr^q\times\Rr^q\setminus\{(0,0)\}\to \Tt^q\times \Rr_+^q
$$ 
be a coordinate change such that $Q_\varepsilon^{-1}(\theta,r)=(x,y)$ with
$$
x_j=\sqrt{\varepsilon r_j} \cos(2\pi\theta_j)
\te{and}
y_j=\sqrt{\varepsilon r_j} \sin(2\pi\theta_j).
$$
We then define on $\Tt^q\times \Rr_+^q$ the map
$$
F_\varepsilon=Q_\varepsilon\circ \widetilde F \circ Q_\varepsilon^{-1}.
$$
Let $\lambda_\varepsilon=Q_\varepsilon^*(r\,d\theta)=\frac{1}{2\pi\varepsilon}(x\,dy-y\,dx)$ be the pull-back by $Q_\varepsilon$ of the form $r\,d\theta$.
Recall that $\widetilde F^*(\lambda_\varepsilon)-\lambda_\varepsilon$ is exact since $\widetilde F$ is a symplectomorphism on a simply connected domain.
So, $F_\varepsilon^*(r\,d\theta)-r\,d\theta$ is exact.
The same applies to any iterate $F_\varepsilon^N$.

Consider also the map 
$$
G_\varepsilon=Q_\varepsilon\circ \psi^{-1}\circ G\circ \psi \circ Q_\varepsilon^{-1},
$$
where $G(z)=\Phi(z)z$ is the first term of $F$. It follows by a simple computation that 
$$
G_\varepsilon(\theta,r)=(\theta+a+\varepsilon\beta r \bmod\Zz^q,r).
$$
Any iterate of $G_\varepsilon$ is a weakly monotonous completely integrable exact symplectomorphism since
$G_\varepsilon^N(\theta,r)=(\theta+Na+N\varepsilon\beta r \bmod\Zz^q,r)$ and
$\det(N\varepsilon\beta)\not=0$ for every $N\in\Nn$.

Notice that the fixed point of $F$ is not in the domain since $Q_\varepsilon$ is not defined at the origin. 
In fact, we focus on the following strip near $\Tt^q\times\{0\}$ so that we can use Theorem~\ref{th:contreras}.
Given $0<\rho<\frac 1{2q}$, define the domain
$$
B_\rho=
\left\{(\theta,r)\in\Tt^q\times\Rr^q\colon 
\sum_{j=1}^q\left(r_j-\frac{1}{2q}\right)^2 <\rho^2
\right\}.
$$

\begin{lemma}[Moser, cf.~\cite{Arnaud1992}]
Let $0<\rho'<\rho$ and $C>0$.
There is $\varepsilon_0>0$ such that
for every $0<\varepsilon<\varepsilon_0$ and $N\in\Nn$ verifying $\varepsilon N<C$, we have
$$
\lim_{\varepsilon\to0}\|F_\varepsilon^N-G_\varepsilon^N\|_{C^1}=0
\te{on}
B_{\rho'}.
$$
\end{lemma}

For sufficiently small $\varepsilon$ recall that $F_\varepsilon^N$ is weakly monotonous because it is $C^1$-close to $G_\varepsilon^N$.
Since we can extend these maps so that the conditions of Theorem~\ref{th:contreras} are fulfilled, we have thus proved the following result by using~\cite[Lemma 8.6]{Arnaud1992}.

\begin{proposition}
Let $\phi\in \emb5$ without degenerate periodic points.
Any 4-elementary weakly monotonous $q$-elliptic periodic point, $1\leq q\leq d$, 
has a 1-elliptic periodic point nearby.
\end{proposition}

\subsection{Essential invariant curves with rational rotation number}

Let $\phi\in \emb2$ and $x$ a 1-elliptic periodic point with period $m$. 
Restrict $f_\phi^m$ to the two-dimensional center manifold of $x$, which is normally hyperbolic.
Use the Birkhoff normal form and the above coordinate change into $\Tt^1\times (0,\delta)$ for some small $\delta>0$.
We can extend this to an area-preserving twist diffeormorphism on the cylinder $\Aa=\Tt^1\times [0,1]$, that can be written in local coordinates as 
$$
\varphi\colon\Aa\to\Aa,
$$
with $\varphi(\theta,0)=(\theta,0)$ corresponding to the periodic point $x$ and $\varphi(\theta,1)\in \Tt^1\times\{1\}$.

Recall the following criterium to find positive topological entropy in the case of area-preserving twist maps (see also ~\cite{boyland-hall-1987,sinisa-2017}).
An essential curve is a non-contractible simple closed curve on $\Aa$. 

\begin{theorem}[Angenent~\cite{angenent-1990,angenent-1992}]
Let $\varphi$  be an area-preserving twist homeomorphism on $\Aa$ with rotation interval $I$ and preserving the boundaries.
If there is $\rho\in I$ such that there are no essential invariant curves with rotation number $\rho$, then $h_{top}(\varphi)>0$.
\end{theorem}

In the case $\varphi$ is a $C^{1+\alpha}$-diffeomorphism, $\alpha>0$, a celebrated result by Katok~\cite[Corollary 4.3]{katok-1980} states that positive topological entropy implies the existence of a hyperbolic periodic point with a transversal homoclinic point.
The hyperbolic periodic point in the central manifold is also a hyperbolic periodic point for the map by~\cite[Lemma 8.6]{Arnaud1992}.
The same holds for the transversal homoclinic point.
Hence, there is a horseshoe for some iterate of the billiard map~\cite[Theorem 6.5.5]{Katok}.

We are therefore left with the case of existence of essential invariant curves of $\varphi$ for any rotation number inside a small interval close to zero. If the rotation number is rational, only two cases can occur: the invariant curve either consists entirely of periodic points or is a heteroclinic chain, i.e. a set formed by finitely many hyperbolic periodic points $x_1,\ldots,x_n$ and heteroclinic connections between them (cf.~\cite{Herman1983,LeCalvez-1991}). In the first case, all periodic points must be degenerate, which is a situation not allowed generically (see Theorem~\ref{thm: generic}). To deal with the second case, first we observe that a hyperbolic periodic point for $\varphi$ is an hyperbolic periodic point for the billiard map. Then we use Theorem~\ref{thm: Donnay} together with parts (1) and (2) of Theorem~\ref{thm: generic} to obtain a convex body arbitrarily close to $\phi$ for which the points $x_1,\ldots,x_n$ are hyperbolic periodic points with transverse heteroclinic points. This property implies the existence of transverse homoclinic points for each periodic point $x_i$, which in turn implies the existence of a horseshoe for some iterate of the billiard map (see~\cite[Section~6]{Katok}). This concludes the proof of Theorem~\ref{thm: ell}.

\subsection{Transverse heteroclinic intersections}

\label{section: Donnay}

The theorem presented here is the multidimensional analog of~\cite[Theorem 1]{Donnay-2005}.

Let $\phi\in  \emb r$, $r\in\{2,3,\ldots,\infty\}$.  
Recall that $ \pi_{1}(x)=p $ for every $ x = (p,v) \in M_\phi $ and $N_\phi(p)\in\Rr^{d+1}$ is the unit normal vector of $\Gamma_\phi$ at $p$ inward-pointing, as defined in section~\ref{sec:coordinates}. 
Given a hyperbolic periodic point $ x \in M_\phi $ of $ f_\phi $, denote by $ W^s_\phi(x) $ and $ W^u_\phi(x) $ the stable and the unstable manifolds of $ x $.

\begin{theorem}\label{thm: Donnay}
Suppose that $f_\phi$ has two hyperbolic periodic points $x$ and $y$ and a heteroclinic point $ z \in W^s_\phi(x) \cap W^u_\phi(y) $ such that $\pi_1(z)\notin \pi_1(O(x)\cup O(y))$.  There exist $u\in C^{\infty}(S,\Rr)$ and $\varepsilon_0>0$ such that for every $0<\varepsilon<\varepsilon_0$ 
the following holds:
\begin{enumerate}
\item$\phi_\varepsilon:=\phi+\varepsilon u\cdot N_\phi(\pi_1(z))\in \emb r$,
\item $\phi_\varepsilon=\phi$ except for a small neighbourhood of $\pi_1(z)$,
\item $O(x)$ and $O(y)$ are hyperbolic periodic orbits of $f_{\phi_\varepsilon}$,
\item $z\in W^s_{\phi_\varepsilon}(x)\cap W^u_{\phi_\varepsilon}(y)$,
\item $T_z W^s_{\phi_\varepsilon}(x)$ and $T_z W^u_{\phi_\varepsilon}(x)$ are transverse. 
\end{enumerate}
\end{theorem}

\begin{proof}
Let $z=(p,v)$ and consider a neighbourhood $ U \subset \Gamma_\phi$ of $ p $ such that $U$ does not intersect $\pi_1(O(x))$,  $\pi_1(O(y))$ and $\pi_1(O(z)\setminus\{z\})$. 
Recall the definition of $ K(x) $ in \eqref{definition of K}.  To stress the dependence on $\phi$ we write $K_\phi(x)$.  By Lemma~\ref{lem:perturbing body},  there exist $ \varepsilon_0 > 0 $ and $u\in C^\infty(S,\Rr)$ such that for every $ 0 \le \varepsilon<\varepsilon_0 $  
the following holds:
\begin{itemize}
\item $\supp(u)\subset \phi^{-1}(U)$
\item $\phi_\varepsilon:=\phi+\varepsilon u\cdot N_\phi(p) \in \emb r$,
	\item $ p \in \Gamma_{\phi_\varepsilon} $ and $ T_p \Gamma_{\phi_\varepsilon} = T_p \Gamma_\phi $,
	\item $ K_{\phi_\varepsilon}(z) = K_\phi(z) +  \varepsilon \Omega_\phi$,
\end{itemize}
where $\Omega_\phi:=-2\langle v,N_\phi(p)\rangle(P_{N_\phi(p)}^v)^*P_{N_\phi(p)}^v$ is an invertible self-adjoint linear operator on $v^\perp$.  
Then
\begin{equation*}
\label{eq:equal_orbits}
f^n_{\phi_\varepsilon}(x) = f_\phi^n(x), f^n_{\phi_\varepsilon}(y) = f_\phi^n(y) \text{ and } f^n_{\phi_\varepsilon}(z) = f_\phi^n(z), \quad   n \in \Zz.	
\end{equation*} 
Hence, $ x $ and $ y $ are hyperbolic periodic points of $ f_{\phi_\varepsilon} $, and $ z $ is a heteroclinic point of $  f_{\phi_\varepsilon} $., i.e., $ z \in W_{\phi_\varepsilon}^s(x) \cap W_{\phi_\varepsilon}^u(y) $. Moreover,
\begin{itemize}
	\item $ T_{f_\phi(z)} W_{\phi_\varepsilon}^s(f_\phi(x))= T_{f_\phi(z)} W_\phi^s(f_\phi(x)) $, 
	\item $ T_{f_\phi^{-1}(z)} W_{\phi_\varepsilon}^u(f_\phi^{-1}(y))= T_{f_\phi^{-1}(z)} W_\phi^u(f_\phi^{-1}(y)) $.
\end{itemize}
By Lemma~\ref{lemma derivative billiard map Jacobi coord}, it follows that
\begin{equation*}
\label{eq:stable_tangent_space}	
 T_{z}W_{\phi_\varepsilon}^s(x)= T_{z} W_\phi^s(x).
\end{equation*}
Since the stable and unstable manifolds of $x$ and $y$ with respect to $f_\phi$ are Lagrangian,  there exist self-adjoint  linear operators $ A_\phi $ and $ B_\phi $ on $N_\phi(p)^\perp$ such that the tangent spaces $ T_{z} W_\phi^s(x) $ and  $ T_z W_\phi^u(y) $ can be written in Jacobi coordinates as follows (see~\cite[Appendix B]{Wojtkowski-1988} and references therein),
\begin{align*} 
& T_{z} W_\phi^s(x) = \left\{(J,(P_{N_\phi(p)}^v)^*A_\phi P_{N_\phi(p)}^vJ) \colon J \in v^\perp\right\}, \\
& T_{z} W_\phi^u(y) = \left\{(J,(P_{N_\phi(p)}^v)^*B_\phi P_{N_\phi(p)}^vJ) \colon J \in v^\perp\right\}.
\end{align*}
Then, again by Lemma~\ref{lemma derivative billiard map Jacobi coord},
\begin{align*}
T_z W_{\phi_\varepsilon}^u(y) & = 
\begin{bmatrix}
I & 0\\
K_{\phi_\varepsilon}(z) & I
\end{bmatrix} 
\begin{bmatrix}
I & 0\\
-K_\phi(z) & I
\end{bmatrix} 
T_z W_\phi^u(y) \\
& = \left\{(J,B_\varepsilon J) \colon J \in v^\perp\right\} 
 \end{align*}
with
\begin{align*}
B_\varepsilon &:= (P_{N_\phi(p)}^v)^*B_\phi P_{N_\phi(p)}^v+ K_{\phi_\varepsilon}(z) - K_\phi (z)\\
&= (P_{N_\phi(p)}^v)^*B_\phi P_{N_\phi(p)}^v + \varepsilon \Omega_\phi\\
&=(P_{N_\phi(p)}^v)^*(B_\phi -2\varepsilon \langle v,N_\phi(p)\rangle I)P_{N_\phi(p)}^v.
\end{align*}
Now, the manifolds $W_{\phi_\varepsilon}^s(x) $ and $ W_{\phi_\varepsilon}^u(y) $ are transversal at $ z $ if and only if $  B_\phi - A_\phi - 2\varepsilon \langle v,N_\phi(p)\rangle I$ is invertible.  This is equivalent to say that $2\varepsilon \langle v,N_\phi(p)$ is not an eigenvalue of $B_\phi-A_\phi$.  If $A_\phi= B_\phi$,  then just choose any $0<\varepsilon<\varepsilon_0$,  otherwise choose any $0< \varepsilon<\min\{\frac{\sigma}{2 \langle v,N_\phi(p)},\varepsilon_0\}$ where $\sigma$ is the least non-zero singular value of $B_\phi-A_\phi$.
\end{proof}

\begin{remark}
Our main result applied to ellipsoids in $\Rr^{n+1}$ implies that there are convex bodies with smooth boundary arbitrarily close to ellipsoids whose billiards map has positive topological entropy. We observe that this conclusion can be obtained in a straightforward manner by applying Theorem~\ref{thm: Donnay} directly to ellipsoids, similarly to what Donnay did for billiards in ellipses~\cite{Donnay-2005}. Indeed, for a billiard in a  ellipsoid $\phi$ of $\Rr^{n+1}$ with a unique major axis, the periodic orbit $\gamma=\{x_1,x_2\}$ corresponding to the major axis of the ellipsoid is hyperbolic, and the periodic points $x_1$ and $x_2$ have heteroclinic connections, i.e. $W^s_{\phi}(x_1)=W^u_{\phi}(x_2)$ and $W^u_{\phi}(x_1)=W^s_{\phi}(x_2)$ \cite[Section 6]{Dels2}. Using Theorem~\ref{thm: Donnay}, we can $C^2$-approximate the ellipsoid $\phi$ by convex bodies $\phi_{\epsilon}$ with smooth boundary such that $x_1$ and $x_2$ are hyperbolic periodic points of period 2 also for $f_{\phi_{\epsilon}}$, and $W^s_{\phi_{\epsilon}}(x_1) $ and $W^u_{\phi_{\epsilon}}(x_2)$ as well $W^u_{\phi_{\epsilon}}(x_1)$ and $W^s_{\phi_{\epsilon}}(x_2)$ have transverse intersections. Hence, $x_1$ and $x_2$ form a heteroclinic chain with transverse heteroclinic points for the billiard map of $\phi_{\epsilon}$, and so each point $x_1$ and $x_2$ must have a transverse homoclinic point. In particular, the billiard inside $\phi_{\epsilon}$ has positive topological entropy. Similar results  were obtained in~\cite{Dels1,Dels2} using a variational approach and Poincar\'e-Melnikov method.
\end{remark}

\section{Franks' lemma for multidimensional billiards}
\label{sec: Franks}


In this section we generalize the Franks' lemma for multidimensional billiards.

Denote by $\Sp(T_xM)$ the set of linear symplectic automorphisms on $T_xM$.

\begin{theorem}\label{Franks}
Let $r\in\{2,3,\ldots,\infty\}$. There is a $C^r$-residual subset $\RR\subset\emb r$ such that for every $\phi\in \RR$,  every 
periodic point $x\in M_\phi$ of $f_\phi$ having period $m\geq4$ and any $\varepsilon>0$ there is $\delta>0$ such that for any $\Pi\in\Sp(T_xM)$ which is $\delta$-close to $Df^m_\phi(x)\colon T_xM\to T_xM$, 
there is $u\in C^\infty( S, \Rr^{d+1})$ which is $\varepsilon$-$C^2$-close to $0$ such that 
\begin{enumerate}
\item $\phi_u:=\phi+u\in \emb k$,
\item $O(x)$ is a periodic orbit of both $f_\phi$ and $f_{\phi_u}$,
\item $Df^m_{\phi_u}(x)=\Pi$,
\item the perturbation $u$ can be chosen to vanish outside an arbitrary small neighbourhood of 4 consecutive reflection points of the periodic orbit $\OO(x)$.
\end{enumerate}
\end{theorem}

\subsection{Some linear algebra}
Let $\Vv$ be a $d$-dimensional real vector space with an inner product $\langle\cdot,\cdot\rangle$.  Let $L(\Vv)$ denote the ring of linear operators on $\Vv$ and $\Sym(\Vv)$ denote the linear space of self-adjoint linear operators on $\Vv$. 
Given $X\in L(\Vv)$,  let $\Psi_X\colon \Sym(\Vv)\to L(\Vv)$ be the linear map
\begin{equation}\label{eq:mapPsi}
\Psi_X(Y)=X^* Y- YX,
\end{equation}
where $X^*$ is the adjoint linear operator of $X$. Clearly,  $\Psi_X(Y)=0$ iff $YX\in \Sym(\Vv)$.
\begin{lemma}\label{lem:simple spectrum}
If $X$ has all its eigenvalues distinct, then there is a subspace $V\subset \Sym(\Vv)$ of dimension $\frac{d(d-1)}{2}$ such that $\Psi_X$ restricted to $V$ is invertible.
\end{lemma}

\begin{proof}
By fixing an orthonormal basis of $\Vv$,  the linear operator $Z:=\Psi_X(Y)$ is represented by a matrix $\hat{Z}:=\hat{X}^\top \hat{Y}-\hat{Y}\hat{X}$ where $\hat{X}$ and $\hat{Y}$ are the matrices of $X$ and $Y$ in the basis of $\Vv$,  and $\hat{X}^\top$ is the transpose matrix of $\hat{X}$.  Using the Kronecker product $\otimes$ we can write $$\hat{Z}=\Phi(\hat{Y}):=(\hat{X}^\top\oplus (-\hat{X}^\top))\mathrm{vec}(\hat{Y}),$$ where $A\oplus B :=( I\otimes A)+(B\otimes I)$ for $A,B\in\Mat(d)$, 
$$
A\otimes B=\begin{bmatrix}
a_{1,1}B&\cdots &a_{1,n} B\\
\vdots&\ddots&\vdots\\
a_{n,1}B&\cdots&a_{n,n}B
\end{bmatrix}\quad\text{and}\quad \mathrm{vec}(A)=\begin{bmatrix}
A_{*,1}\\
\vdots\\
A_{*,n}
\end{bmatrix}.
$$
It is known that $A\oplus B$ has eigenpair $(\lambda_A+\lambda_B,v_A\otimes v_B)$ if $(\lambda_A,v_A)$ is an eigenpair of $A$ and $(\lambda_B,v_B)$ is an eigenpair of $B$ (see~\cite[Theorem 4.4.5]{HJ91}). Any eigenvalue of $A\oplus B$ arises in this way. Since $\hat{X}$ has $d$ distinct eigenvalues, the kernel of $\Phi$ has dimension $d$. Moreover,
$$
\mathrm{Ker}(\Phi)=\{v\otimes v\,\colon\, v\text{ is a left eigenvector of }\hat{X} \}\subset \Sym(d),
$$
where we make the identification $v\otimes v \equiv \begin{bmatrix}
v_1 v&\cdots &v_{d} v
\end{bmatrix}\in \Sym(d)$. 
Therefore,  $\mathrm{Ker}(\Psi_X)\subset \Sym(\Vv)$ and $\dim \mathrm{Ker}(\Psi_X) = d$.

Now consider any subspace $V\subset \Sym(\Vv)$ such that $\Sym(\Vv) = \mathrm{Ker}(\Psi_X)\oplus V$.  Clearly, $\Psi_X$ restricted to $V$ is invertible and $$\dim{V}=\dim \Sym(\Vv)-\dim \mathrm{Ker}(\Psi_X) = \frac{d(d-1)}{2}.$$
\end{proof}

On $\Vv\times\Vv$ we introduce the canonical symplectic form 
$$
\Omega(w_1,w_2)=\langle y_1,z_2\rangle-\langle y_2,z_1\rangle,
$$    
where $w_i=(y_i,z_i)\in\Vv\times\Vv$.  We denote by $\Sp(\Vv\times \Vv)$ the group of symplectic linear operators on $\Vv\times\Vv$.  

Let $C,F\in\Sp(\Vv\times \Vv)$ given by 
$$
C(K)=\begin{bmatrix}
I&0\\K&I
\end{bmatrix}\quad\text{and}\quad
F(\tau)=\begin{bmatrix}
I&\tau I\\0&I
\end{bmatrix}
$$
where $K\in \Sym(\Vv)$,  $\tau\in\Rr$ and $I$ is the identity operator on $\Vv$.  The derivative of the billiard map in Jacobi coordinates has the form,
$$
A(K,\tau):=C(K)F(\tau) = \begin{bmatrix} I&\tau I\\K&I+\tau K\end{bmatrix}\in \Sp(\Vv\times \Vv).
$$  
Throughout this section fix $\tau_1,\tau_2,\tau_3,\tau_4\in\Rr\setminus\{0\}$ and define the map $$B\colon \Sym(\Vv)^4\to\Sp(\Vv\times \Vv)$$ by
\begin{equation}\label{map B}
(K_1,K_2,K_3,K_4)\mapsto A(K_4,\tau_4)A(K_3,\tau_3)A(K_2,\tau_2)A(K_1,\tau_1).
\end{equation}
We also define the following linear operators on $\Vv$:
\begin{equation}\label{def Omega}
\Omega(K_2):=K_2+(1/\tau_2+1/\tau_3)I
\end{equation}
and
\begin{equation}\label{def Delta}
\Delta(K_2,K_3):=(K_3+(1/\tau_3+1/\tau_4)I)(K_2+(1/\tau_2+1/\tau_3)I),
\end{equation} 
where $K_2$ and $K_3$ belong to $\Sym(\Vv)$. 

\begin{defn}\label{def F property}
We say that $(K_2,K_3)\in \Sym(\Vv)^2$ satisfy the \textit{$F$-property} if $\Omega(K_2)$ is invertible and $\Delta(K_2,K_3)$ has $d$ distinct eigenvalues.
\end{defn}
\medskip

The following result is crucial in the proof of Franks' lemma.
\begin{lemma}\label{lem:franks}
If $K=(K_1,K_2,K_3,K_4)\in \Sym(\Vv)^4$ such that $(K_2,K_3)$ satisfy the $F$-property, then $DB(K)$ has full rank and $B$ is a submersion at $K$.
\end{lemma}

\begin{proof}
Let $\Omega=\Omega(K_2)$ and $\Delta=\Delta(K_2,K_3)$ as defined in \eqref{def Omega} and \eqref{def Delta}, respectively.  Fix a basis $\{E_{i,j}\}_{1\leq i\leq j\leq d}$ for $\Sym(\Vv)$.  By Lemma~\ref{lem:simple spectrum},  there is a subspace $V$ of $\Sym(\Vv)$ of dimension $\frac{d(d-1)}{2}$ such that $\Psi_\Delta$ (as defined in \eqref{eq:mapPsi}) restricted to $V$ is invertible.  Let $\{V_j\}$ denote a basis for $V$ and define,
\begin{align*}
X^{(j)}&:=DB(K)(V_j,0,0,0),\\
Y^{(i,j)}&=DB(K)(0,E_{i,j},0,0),\\
Z^{(i,j)}&=DB(K)(0,0,E_{i,j},0),\\
W^{(i,j)}&=DB(K)(0,0,0,E_{i,j}).
\end{align*}
Since
\begin{equation*}\label{eq:match dimensions}
\dim{V}+ 3\times \dim{\Sym(\Vv)}=\dim{\Sp(\Vv)},
\end{equation*}
to prove that $DB(K)$ has full rank it is sufficient to show that the tangent vectors defined above are linearly independent, i.e., 
$$
\sum_{j}\alpha_{j} X^{(j)}+\sum_{i\leq j}\beta_{i,j} Y^{(i,j)}+\sum_{i\leq j}\gamma_{i,j} Z^{(i,j)}+\sum_{i\leq j}\theta_{i,j} W^{(i,j)}=0
$$
implies that $\alpha_{j}=\beta_{i,j}=\gamma_{i,j}=\theta_{i,j}=0$. Taking into account that 
$$
DA(K,\tau)(S)=\begin{bmatrix}
0&0\\S&\tau S
\end{bmatrix},\quad S\in \Sym(\Vv)
$$
we compute,
\begin{align*}
X^{(j)} &= \begin{bmatrix}
0&0\\V_{j}&\tau_4 V_{j}
\end{bmatrix}\begin{bmatrix}I&\tau_3 I\\K_3&I+\tau_3 K_3\end{bmatrix}\begin{bmatrix}I&\tau_2 I\\K_2&I+\tau_2 K_2\end{bmatrix}\begin{bmatrix}I&\tau_1 I\\K_1&I+\tau_1 K_1\end{bmatrix},\\
Y^{(i,j)} &= \begin{bmatrix}I&\tau_4 I\\K_4&I+\tau_4 K_4\end{bmatrix}\begin{bmatrix}
0&0\\E_{i,j}&\tau_3 E_{i,j}
\end{bmatrix}\begin{bmatrix}I&\tau_2 I\\K_2&I+\tau_2 K_2\end{bmatrix}\begin{bmatrix}I&\tau_1 I\\K_1&I+\tau_1 K_1\end{bmatrix},\\
Z^{(i,j)} &=  \begin{bmatrix}I&\tau_4 I\\K_4&I+\tau_4 K_4\end{bmatrix}\begin{bmatrix}I&\tau_3 I\\K_3&I+\tau_3 K_3\end{bmatrix}\begin{bmatrix}
0&0\\E_{i,j}&\tau_2 E_{i,j}
\end{bmatrix}\begin{bmatrix}I&\tau_1 I\\K_1&I+\tau_1 K_1\end{bmatrix},\\
W^{(i,j)} &= \begin{bmatrix}I&\tau_4 I\\K_4&I+\tau_4 K_4\end{bmatrix}\begin{bmatrix}I&\tau_3 I\\K_3&I+\tau_3 K_3\end{bmatrix}\begin{bmatrix}I&\tau_2 I\\K_2&I+\tau_2 K_2\end{bmatrix}\begin{bmatrix}
0&0\\E_{i,j}&\tau_1 E_{i,j}
\end{bmatrix}.
\end{align*}
Define $\hat{X}^{(j)}:=A(K_4,\tau_4)^{-1}X^{(j)}F(\tau_1)^{-1}$, $\hat{Y}^{(i,j)}:=A(K_4,\tau_4)^{-1}Y^{(i,j)}F(\tau_1)^{-1}$, etc. Using the fact that
$$
A(K,\tau)^{-1}=\begin{bmatrix}I+\tau K&-\tau I\\-K&I\end{bmatrix}\quad\text{and}\quad F(\tau_1)^{-1}=F(-\tau_1)
$$
we find that
\begin{align*}
\hat{X}^{(j)}&=\begin{bmatrix}
-\tau_4 X^{(j)}_{2,1}&-\tau_4 X^{(j)}_{2,2}\\
X^{(j)}_{2,1}&X^{(j)}_{2,2}
\end{bmatrix}\\
\hat{Y}^{(i,j)}&=\begin{bmatrix}
0&0\\
\hat{Y}^{(i,j)}_{2,1}& E_{i,j}((\tau_2+\tau_3)I+\tau_2\tau_3K_2)
\end{bmatrix}\\
\hat{Z}^{(i,j)}&=\begin{bmatrix}
\tau_3E_{i,j}(I+\tau_2K_1)&\tau_2\tau_3 E_{i,j}\\
(I+\tau_3K_3)E_{i,j}(I+\tau_2K_1)& \tau_2(I+\tau_3K_3)E_{i,j}
\end{bmatrix}\\
\hat{W}^{(i,j)}&=\begin{bmatrix}
((\tau_2+\tau_3)I+\tau_2\tau_3K_2)E_{i,j}&0\\
(\tau_2 K_3+(I+\tau_3K_3)(I+\tau_2K_2))E_{i,j}& 0
\end{bmatrix}
\end{align*}
where $\hat{Y}^{(i,j)}_{2,1}$, $X_{2,1}^{(j)}$ and $X_{2,2}^{(j)}$ can be explicitly computed, but only the expression for $X_{2,2}^{(j)}$ will be needed,
\begin{align*}
X_{2,2}^{(j)}&=V_j(\tau_2 I +(I+\tau_2K_2)\tau_3+\tau_4(I+\tau_3K_3)(I+\tau_2K_2)+\tau_2\tau_4K_3)\\
&=\tau_2\tau_3\tau_4V_j\left(\Delta- \frac{1}{\tau_3^2} I\right).
\end{align*}
Hence, in the $(1,2)$-th block we have the equation
$$
-\tau_2\tau_3\tau_4^2\left(\sum_{j}\alpha_{i}V_j\right)\left(\Delta- \frac{1}{\tau_3^2} I\right)+\tau_2\tau_3\sum_{i\leq j}\gamma_{i,j}E_{i,j}=0.
$$
Thus, we must have  $H\Delta\in \Sym(\Vv)$ where $H=\sum_{j}\alpha_{j}V_j\in V$. Equivalently, $\Psi_\Delta(H)=0$. By Lemma~\ref{lem:simple spectrum}, $\Psi_\Delta$ restricted to $V$ is invertible, hence $H=0$, i.e., $\alpha_{j}=0$ for every $j$. This also implies that $\gamma_{i,j}=0$ for every $i\leq j$. Now, looking at the $(2,2)$-th block we have the equation $\tau_2\tau_3\sum_{i\leq j}\beta_{i,j}E_{i,j}\Omega=0$. Since $\Omega$ is invertible, we conclude that $\beta_{i,j}=0$ for every $i\leq j$. Finally, looking at the $(1,1)$-th block we get, by the same argument, that $\theta_{i,j}=0$ for every $i\leq j$. 
This concludes the proof.
\end{proof}

\begin{remark}
When $d=1$, we have $K_2=\frac{2k_2}{\cos\theta_2}$ where 
$k_2$ is the curvature of the billiard table at the collision point $p_2$ and $\theta_2$ the reflection angle at $p_2$, and the pair $(K_2,K_3)$ satisfies the $F$-property iff $\Omega(K_2)$ is invertible, i.e. iff 
$$
K_2\neq-\left(\frac{1}{\tau_2}+\frac{1}{\tau_3}\right).
$$
This is precisely the non-focusing condition for the three consecutive collisions $\{p_1,p_2,p_3\}\subset \Gamma$ used in~\cite{visscher-2015}.
\end{remark}

\begin{lemma}\label{lem:Franks generic}
The set of pairs $(K_2,K_3)\in \Sym(\Vv)^2$ that satisfy the $F$-property is open and dense in $\Sym(\Vv)^2$.
\end{lemma}
\begin{proof}
Clearly, the set of invertible linear operators $\GL(\Vv)$ is open and dense in $L(\Vv)$.  Moreover, the set $\mathcal{S}(\Vv)$ of linear operators with simple spectrum, i.e., distinct eigenvalues, is also open and dense in $L(\Vv)$.  Thus $\mathcal{N}:=\GL(\Vv)\cap \mathcal{S}(\Vv)$ is open and dense.  Now, consider the map $\varphi\colon(A_1,A_2)\mapsto A_1 A_2$ mapping $\GL(\Vv)^2$ to $\GL(\Vv)$.  By continuity, $\varphi^{-1}(\mathcal{N})$ is open.   Moreover,  $\varphi^{-1}(\mathcal{N})$ is dense. Indeed, suppose by contradiction that $\varphi^{-1}(\mathcal{N})\cap (\mathcal{V}_1\times \mathcal{V}_2)=\emptyset$ for some open subsets $\mathcal{V}_i$ of $\GL(\Vv)$.  Given $A\in\mathcal{V}_1$,  the map $\varphi_A\colon \GL(\Vv) \to \GL(\Vv)$ defined by $\varphi_A(B)=\varphi(A,B)$ is a continuous surjective homomorphism,  thus it is open by the
open mapping theorem for locally compact groups (see e.g.  \cite[Theorem~5.29]{HR63}).  This implies that $\varphi_A(\mathcal{V}_2)$ is open in $\GL(\Vv)$, and since $\mathcal{N}$ is dense,  we can find $C\in \mathcal{N}\cap \varphi_A(\mathcal{V}_2)$ and $B\in \mathcal{V}_2$ such that $AB=C$, which contradicts the fact that $\varphi^{-1}(\mathcal{N})$ does not intersect $\mathcal{V}_1\times \mathcal{V}_2$.
Therefore,  the set 
$$
\MM_1=\{ (K_2,K_3)\in  \Sym(\Vv)^2\colon \Delta(K_2,K_3)\text{  has simple spectrum}\}
$$
is open and dense in  $\Sym(\Vv)^2$. By the definition of $\Omega(K_2)$ is is also clear that 
$$
\MM_2=\{ K_2\in \Sym(\Vv)\colon  \Omega(K_2)\text{ is invertible}\}
$$
is also open and dense in $\Sym(\Vv)$.  Hence,  $\mathcal{M} = \MM_1 \cap (\MM_2\times  \Sym(\Vv))$ is open and dense as we wanted to show.
\end{proof}

\subsection{Proof of Theorem~\ref{Franks}}
\label{sec:linear independence}

Recall $K(x)\colon v^\perp\to v^\perp$ from \eqref{definition of K} where $x=(p,v)\in M$. We say that an orbit segment $\gamma=\{x_i\}_{i=0}^{n}\subset \Int(M)$ of $f$ is \textit{$F$-admissible} if $n\geq3$,  $\pi_1(\gamma)$ consists of $n+1$ points in $\Gamma_\phi$ and there is $2\leq k <n$ such that  $(K(x_k),R_{p_{k+1}}^{-1}K(x_{k+1}) R_{p_{k+1}})$ satisfy the $F$-property as linear operators in $\Sym(v_k^\perp)$.

\begin{theorem}\label{thm:baby franks}
Let $r\in\{2,3,\ldots,\infty\}$.  Let $\phi\in \emb r$ and consider an $F$-admissible orbit segment $\gamma=\{x_i\}_{i=0}^{n}\subset \Int(M_\phi)$ of $f_\phi$.  For every $\varepsilon>0$ there is $\delta>0$ such that for any symplectic linear map $A\colon T_{x_0}M_\phi \to T_{x_n}M_\phi $ which is $\delta$-close to $Df^n_\phi(x_0)$, 
there is $u\in C^\infty(S,\Rr^{d+1})$ which is $\varepsilon$-$C^2$-close to $0$ such that 
\begin{enumerate}
\item $\phi_u:=\phi+u\in \emb k$,
\item $\gamma$ is an orbit segment of both $f_\phi$ and $f_{\phi_u}$,
\item $Df^n_{\phi_u}(x_0)=A$,
\item the perturbation $u$ can be chosen to vanish outside a arbitrary small neighbourhood of four consecutive points in $\pi_1(\gamma)$.
\end{enumerate}
\end{theorem}

\begin{proof}
Let $x_i=(p_i,v_i)$ for $i=0,\ldots,n$.  According to Lemma~\ref{lemma derivative billiard map Jacobi coord},  the derivative of $f_\phi^n$ at $x_0$ in Jacobi coordinates is the linear map $Df_\phi^n(x_0)\colon v_0^\perp\times v_0^\perp \to v_n^\perp\times v_n^\perp$ given by 
\begin{equation}\label{eq: product derivatives}
Df_\phi^n(x_0)=C(x_n)U(p_n)F(\tau_{n})\cdots C(x_1)U(p_1)F(\tau_1)
\end{equation} 
where
$$
C(x_i)=\begin{bmatrix}
I&0\\K(x_i)&I
\end{bmatrix},  U(p_i)=\begin{bmatrix}
R_{p_i}&0\\0&R_{p_i}
\end{bmatrix}, F(\tau_i)=\begin{bmatrix}
I&\tau_i I\\0&I
\end{bmatrix},
$$
and $\tau_i=\|p_i-p_{i-1}\|>0$ for $i=1,\ldots, n$.  Also recall that $K(x_i)$ is a self-adjoint linear operator on $v_i^{\perp}$ and $R_{p_i}:v_{i-1}^\perp\to v_{i}^\perp$ is the reflection about the hyperplane $N(p_i)^\perp$ which maps $v_{i-1}^\perp$ isometrically onto $v_{i}^\perp$.       

We can rewrite the product \eqref{eq: product derivatives} in the following way
$$
Df_\phi^n(x_0) = U(p_n)\cdots U(p_1)A(K_n,\tau_n)\cdots A(K_1,\tau_1)
$$
where 
$$
A_i:=A(K_i,\tau_i)=\begin{bmatrix}
I&\tau_i I\\K_i&I+\tau_i K_i
\end{bmatrix} \in \Sp(v_0^\perp\times v_0^\perp)
$$
and $K_i := (R_{p_i}\cdots R_{p_1})^{-1} K(x_i)R_{p_i}\cdots R_{p_1}\in\Sym(v_0^\perp)$.  By hypothesis,  there is $2\leq k<n$ such that $(K_k,  K_{k+1})$ satisfy the $F$-property.  Hence,  by Lemma~\ref{lem:franks},  the map $B$ defined in \eqref{map B} with $\Vv:=v_0^\perp$ is a submersion at $K:=(K_{k-1},K_k,K_{k+1},K_{k+2})$.  This shows that for any $\varepsilon>0$ there is a $\delta>0$ such that any symplectic linear map $\Pi\in \Sp(v_0^\perp\times v_0^\perp)$ which is $\delta$-close to $A_n\cdots A_1$ may be realized as
$$
\Pi=A_n\cdots A_{k+3}B(\hat{K}) A_{k-2}\cdots A_{1}
$$
by choosing $\hat{K}\in\Sym(v_0^\perp)^4$ which is $\varepsilon$-close to $K$.  By Lemma~\ref{lem:perturbing body},  for each $i\in\{k-1,k,k+1,k+2\}$,  we can realize the curvature $\hat{K}_i$ at $p_i$ by choosing appropriate smooth functions $u_{i}\in C^{\infty}(S,\Rr)$ which are $\varepsilon$-$C^2$-close to $0$ and whose supports are contained in arbitrarily small neighbourhoods of $p_i$.  Finally,  we define the perturbation $u=\sum_{i=k-1}^{k+2}u_i N_\phi(p_i)$ and the rest of the properties follow.
\end{proof}

Given $m\geq4$,  denote by $\YY_m$ the set of $\phi\in\emb r$ such that $f_\phi$ has only a finite number of periodic orbits of period less than or equal to $m$,  each periodic orbit passes only once through each of its reflection points, and any two different periodic orbits have no common reflection point.  By Theorem~\ref{thm: generic},  $\YY_m$ is a $C^r$-residual subset of $\BB^r$.  Now, denote by $\RR_m$ the subset of $\YY_m$ consisting of those $\phi\in \RR_m$ such that every periodic orbit of $f_\phi$ of period $\geq 4$ is $F$-admissible.  Being $F$-admissible means that the pair of curvature operators at two consecutive collision points belongs to an open and dense set (Lemma~\ref{lem:Franks generic}).  Since,  for each $\phi\in\RR_m$, the billiard map $f_\phi$ has only a finite number of periodic orbits of period less than $m$,  by continuity of $\phi\mapsto L_\phi$ we conclude that $\RR_m$ is $C^r$-open relative to $\YY_m$.   Moreover, by Lemma~\ref{lem:perturbing body},  the set $\RR_m$ is also $C^r$-dense.  Hence,  $\RR_m$ is a $C^r$-residual set and the intersection $\RR=\bigcap_{m\geq4}\RR_m$  is also $C^r$-residual.   Finally, Theorem~\ref{Franks} follows from Theorem~\ref{thm:baby franks}.\qed

\section{Proof of Theorem~\ref{thm: hyp}}\label{sec: hyp}

Let $\mathcal{R}_1$ be the $C^\infty$-residual set of Theorem~\ref{Franks} and $\mathcal{R}_2$ the one of Theorem~\ref{thm: generic}.
Hence, one can take a $C^\infty$-residual 
$$
\mathcal{R}=\mathcal{R}_1\cap \mathcal{R}_2
$$
on which there are an infinite number of periodic orbits with arbitrarly large periods, all nondegenerate.

Consider the set of periodic points $\Per_{m_d}(\phi)$ with period $\geq m_d$, and denote its closure by
$$
\Lambda=\Lambda(\phi):=\overline{\Per_{m_d}(\phi)}.
$$

\begin{theorem}\label{thm: lambda uf}
If $\phi\in \RR\cap\FF^2$, then 
$\Lambda(\phi)$ is hyperbolic.
\end{theorem}

\begin{proof}
Since $\phi\in\FF^2$, the periodic points are hyperbolic, i.e. 
$$
P:=\Per_{m_d}(\phi)\subset H(\phi).
$$
We will apply the general result on the hyperbolicity of families of periodic sequences of bounded symplectic linear maps given in~\cite[Theorem 8.1]{contreras-am-2010}.
It is a symplectic version of the Ma\~n\'e dichotomy on uniform dominated splitting versus trivial spectrum.

For $x\in P$ and $n\in\Zz$ choose an orthonormal symplectic basis $B(x,n)$ of the tangent space $T_{f_\phi^n(x)}M$.
Thus the tangent map 
$$
Df_\phi(f_\phi^n(x))\colon T_{f_\phi^n(x)}M \to T_{f_\phi^{n+1}(x)}M
$$
is represented in the basis $B(x,n)$ by a matrix denoted by $\xi^x_n\in\Sp(\Rr^{2d})$, a symplectic linear map on $\Rr^{2d}$.
Consider the family $\xi=(\xi^x)_{x\in P}$ of sequences $\xi^x=(\xi^x_n)_{n\in\Zz}$.

For each periodic point $x$ with period $m$ we have periodicity of $\xi$ in the sense that
$$
\xi^x_{n+m}=\xi^x_{n},
\quad
n\in\Zz,
$$
because they represent the matrices of $Df_\phi(f^{n+m}(x))=Df_\phi(f^n(x))$ as before.
In addition, we have hyperbolicity of $\xi$, i.e.
$$
\xi^x_{m-1}\dots\xi^x_0
$$
is hyperbolic as this product of matrices corresponds to $Df_\phi^m(x)$ which is hyperbolic.
Notice that all maps $\|\xi^x_n\|$ are uniformly bounded as the manifold is compact.

Suppose that we can perturb $Df_\phi^m(x)$ within $\Sp(T_xM)$ to a non-hyperbolic map $\Pi$ (there is an eigenvalue of modulus 1). 
Use Franks' lemma Theorem~\ref{Franks} to find $\widetilde\phi$ which is $C^2$-close to $\phi$ such that $Df_{\widetilde\phi}^m(x)=\Pi$. This means that $\phi\not\in\FF^2$ since $x$ is a non-hyperbolic period point for $\widetilde\phi$ with period $\geq m_d$. 

So, $\xi$ is a stably hyperbolic family of periodic sequences of bounded symplectic linear maps.
By ~\cite[Theorem 8.1]{contreras-am-2010}, $\xi$ is uniformly hyperbolic.
Therefore, on $P$ the tangent map $Df_\phi$ has an invariant spliting as in \eqref{uh} and $P$ is hyperbolic.
By the continuity of the splitting, this extends to the closure $\Lambda$ of $P$.
\end{proof}

\begin{remark}
There is a known alternative approach to prove Theorem~\ref{thm: lambda uf} following the strategies used in several contexts, cf.~\cite{BRT0,BDP2003,SX2006,HT2006,BDT2021}. 
The main tool is the use of a symplectic version of~\cite[Corollary~2.18]{BGV} as in ~\cite{BCF2017}. 
It can be used to show the following dichotomy: there is either an uniform dominated splitting on the set of periodic orbits having sufficiently large periods, or else a perturbation of the tangent map with trivial spectrum.
Our version of the Franks' lemma (Theorem~\ref{Franks}) then allows us to realize any small perturbation of the tangent map by a $C^2$-perturbation of the body having a degenerate periodic point. However that is not allowed in $\FF^2$.
Therefore, $\Per_{m_d}(\phi)$ has a uniform dominated splitting, which in the symplectic case means that $\Per_{m_d}(\phi)$ is in fact partially hyperbolic~\cite[Theorem 11]{BV2204}. 
Restrict the tangent map to the central subspace of the splitting.
Using a Jordan normal form for symplectic matrices \cite{LK,Gutt}, again by the above dichotomy one gets a partially hyperbolic splitting.
So, the dimension of the stable and unstable subspaces of the original tangent map on $\Per_{m_d}(\phi)$ increases. Repeating this procedure leads us to the conclusion that $\Per_{m_d}(\phi)$ is in fact hyperbolic, and the same holds for $\Lambda$.
\end{remark}

In view of the previous theorem and the Spectral Decomposition Theorem~\cite[pp.385]{Robinson}, the set $\Lambda$ is a union of finitely many pairwise disjoint basic hyperbolic sets. Since $\Lambda$ is infinite by Theorem~\ref{thm: infinito}, at least one of those basic sets must be nontrivial.

\section*{Acknowledgements}

The authors were partially funded by
the project ``New trends in Lyapunov exponents'' PTDC/MAT-PUR/29126/2017.
MB was also partially funded by the projects ``Means and Extremes in Dynamical Systems'' PTDC/MAT-PUR/4048/2021 and UIDB/00144/2020,
JLD and JPG by the project CEMAPRE - UID/MULTI/00491/2019 and
MJT by the projects UIDB/00013/2020 (DOI: 10.54499/UIDB/00013/2020) and UIDP/00013/2020 (DOI: 10.54499/UIDP/00013/2020).
All these projects were financed by Funda\c c\~ao para a Ci\^encia e a Tecnologia, Portugal.
GDM acknowledges the MIUR Excellence Department Project awarded to the Department of Mathematics, University of Pisa, CUP I57G22000700001 and the PRIN Project 2022NTKXCX ``Stochastic properties of dynamical systems'', funded by the Ministry of University and Scientific Research of Italy.


\bibliographystyle{abbrv}
\bibliography{rfrncs}

\end{document}